\documentclass[11pt]{amsart}

\usepackage{amssymb}

\newtheorem{theorem}{Theorem}[section]
\newtheorem{lemma}[theorem]{Lemma}

\newtheorem{proposition}[theorem]{Proposition}

\newtheorem{remark}[theorem]{Remark}

\numberwithin{equation}{section}
\newcommand{\Ai}{\text{Ai\,}}

\newcommand{\re}{\text{Re\,}}

\begin{document}
\setcounter{page}{1}

\thanks{Supported by the Swedish  
Research Council (VR) and grant KAW 2010.0063 from the Knut and Alice Wallenberg
Foundation}
%\dedicatory{}

\title[Non-colliding Brownian motions and the tacnode]
{Non-colliding Brownian motions and the extended tacnode process}
\author[K.~Johansson]{Kurt Johansson}

\address{
Department of Mathematics,
KTH Royal Institute of Technology,
SE-100 44 Stockholm, Sweden}

\email{kurtj@kth.se}

\begin{abstract}

We consider non-colliding Brownian motions with two starting points and
two endpoints. The points are chosen so that the two groups of Brownian motions
just touch each other, a situation that is referred to as a tacnode. The extended
kernel for the determinantal point process at the tacnode point is computed using
new methods and given in a different form from that obtained for a single time in previous work
by Delvaux, Kuijlaars and Zhang. The form of the extended kernel is also
different from that obtained for the extended tacnode kernel in another model
by Adler, Ferrari and van Moerbeke. We also obtain the correlation kernel 
for a finite number of non-colliding Brownian motions starting at
two points and ending at arbitrary points.

\end{abstract}

\maketitle

\section{Introduction and results}\label{sect1}
\subsection{Introduction}\label{sect1.1}
Non-colliding Brownian motions have been 
much studied since they give rise to interesting classes
of determinantal point processes, see e.g. \cite{KM59}, \cite{FN98}, \cite{Jo01}, \cite{Jo05} and \cite{KT07}. 
Scaling limits of these point processes lead to universal limiting
determinantal point processes which also occur in other contexts, e.g. random matrix theory, 
random growth models and certain random tiling (dimer) models. In the latter cases, though we are typically
dealing with non-colliding discrete random walk type models, we can expect continuum
scaling limits for the discrete models to be the same as those
for non-colliding Brownian motions.
Therefore this serves as a natural test model.
The basic point processes obtained are the sine-kernel process (see e.g. \cite{BH97}, \cite{Jo01}, \cite{OR03}),
the Airy kernel process (see e.g. \cite{Fo93}, \cite{PS02}, \cite{OR03}, \cite{Jo03}, \cite{FS03}, \cite{AvMD08}), and the Pearcey
kernel process (see e.g.  \cite{BH98}, \cite{AvM05}, \cite{OR07}, \cite{TW06}, \cite{AOvM10}), and their extended versions. 
If we consider a global asymptotic regime in which the Brownian paths are confined
to a certain geometrical shape, the sine kernel process appears in the bulk (i.e. in the interior of the shape),
the Airy kernel process appears at a typical boundary point of the shape, and the Pearcey process
appears when we have a cusp.
In this paper we investigate a new limiting determinantal
point process that has only been studied recently called the {\it tacnode point process}. This type of process
can be obtained in situations where we have several starting and several endpoints for non-colliding
Brownian motions, see \cite{DK08}, \cite{DK09}, \cite{AFvM08}, \cite{AvMV10}. 

Consider $n$ Brownian motions starting at $0$ at time $t=0$ and ending at $0$ at time $t=1$.
At time $t$ they will lie approximately between $\pm\sqrt{4nt(1-t)}$. Thus, if we have two groups of non-colliding 
Brownian motions in the time interval $0\le t\le 1$, one group of $n$ particles starting and ending at $\sqrt{n}$ and the other,
also with $n$ particles,
starting and ending at $-\sqrt{n}$, then at time $t=1/2$ the two groups will just touch each other at the origin. If
we rescale by $\sqrt{n}$ in the vertical direction, then for large $n$ the two groups of non-colliding Brownian motions
will form ellipses which are tangent at the origin. We can think of this as two colliding Airy processes and we can
expect the same scalings as for the Airy process.
This situation is called a tacnode and the new limiting determinantal
point process that we expect to see in a neighbourhood of the origin is called the tacnode point
process. This point process is determinantal with a kernel that is more complicated than
the sine, Airy and Pearcey kernels. It is not expressible as a single or double contour integral
with elementary functions.
For Brownian motions the tacnode process has been analyzed recently at a single time by Delvaux, Kuijlaars and
Zhang, \cite{DKZ10}, using a $4\times 4$ Riemann-Hilbert problem. In another type of model, which involves, instead of
Brownian motions, Markov chains with discrete space and continuous time, similar to those that occur for
the polynuclear growth model, \cite {PS02}, Adler, Ferrari and van Moerbeke, \cite{AFvM10} have obtained an extended tacnode kernel.
If we consider this kernel at a single time we expect it, by universality, to be the same as that obtained in
\cite{DKZ10}. However the expressions obtaained are quite different and it is not immediate that they give rise
to the same correlation kernels, although we expect this to be the case. In this paper we obtain the
extended tacnode kernel for the Brownian motion model. The approach is not the same as those in \cite{DKZ10} or in
\cite{AFvM10}. It does not use the Riemann-Hilbert method or orthogonal polynomials, but just as in
\cite{AFvM10} Toeplitz determinants and the Geronimo-Case/Borodin-Okounkov identity enter into the
computations. The formula found for the extended tacnode kernel, see (\ref{1.26}) below, has similarities with that obtained
in \cite{AFvM10}, but we have not so far been able to show directly that they give rise to the same correlation functions.
Indirectly, that follows from the analysis in \cite{AJvM11}. Although the formulas involve similar objects, the structure
of the expressions is different. Also, for a single time we cannot directly relate
the kernel to that obtained in \cite{DKZ10}, although, since they concern the same model, we know that
they must define the same correlation functions. 

As in previously obtained scaling limits from non-colliding Brownian motions we expect the
tacnode process to be a natural, universal, scaling limit that should occur also in other contexts.
In joint work in progress with Adler and van Moerbeke we will show that the tacnode process
can be obtained in a random domino model called the double Aztec diamond, \cite {AJvM11}, a
certain extension of the classical Aztec diamond random tiling model. Note that there is also another type of
tacnode limit that has been studied, see \cite{BD10}. In this paper we furthermore obtain formulas for the correlation
kernel, see (\ref{1.14}), when the non-colliding Brownian motions can be divided into two groups, $n$ starting at a
point $a_1$ and ending at arbitrary points, and $m$ starting at a point $a_2$ and ending at arbitrary points.
This generalizes the much studied case when we have a single starting point and arbitrary endpoints,
which has found interesting applications in random matrix theory, see e.g. \cite{BH97}, \cite{Jo01}, \cite{Er10}.

\bigskip
{\bf Acknowledgement.} I thank Pierre van Moerbeke for drawing my attention to the tacnode problem for
non-colliding Brownian motions and Mark Adler and Pierre van Moerbeke for discussions. My sincere thanks 
to Arno Kuijlaars
for helpful discussions and encouragement, as well as for pointing out the usefulness of the Geronimo-Case/Borodin-Okounkov
formula in the present setting. Thanks also to Anthony Metcalfe for comments. I would also like to thank an anonymous
referee for very detailed reading and helpful comments.
Part of this work was carried out while participating in the Random Matrices and Applications
program at MSRI during the fall of 2010. I thank MSRI for hospitality and support.

\subsection{Results}\label{sect1.2}
Consider $N$ non-colliding Brownian particles starting at \linebreak 
$\mu_1,\dots,\mu_N$ at time 0 and ending
at $\nu_1,\dots,\nu_N$ at time 1. We assume that $\mu_1<\dots<\mu_N$ and $\nu_1<\dots<\nu_N$.
Let $(t_r, x_j^{(r)})$, $r=1,\dots,\ell$, $1\le j\le N$,
be the configuration of the particles at times $t_1<t_2<\dots<t_{\ell}$, i.e. at time $t_r$ 
the particles are at the positions $x_1^{(r)},\dots,x_N^{(r)}$. This forms an (extended) 
determinantal point process with kernel, \cite{EM97}, \cite{Jo03}, \cite{Jo05},
\begin{equation}\label{1.1}
\mathcal{L}(s,u,t,v;\mu,\nu)=-p_{t-s}(u,v)+\sum_{j,k=1}^Np_{1-s}(u,\nu_k)(A^{-1})_{kj}p_t(\mu_j,v),
\end{equation}
where
\begin{equation}
p_t(x,y)=\frac{1}{\sqrt{2\pi t}}e^{-(x-y)^2/2t}
\notag
\end{equation}
if $t>0$, $p_t(x,y)=0$ if $t\le 0$, and
\begin{equation}
A=(p_1(\mu_i,\nu_j))_{1\le i,j\le N}.
\notag
\end{equation}
Fix $a_1<a_2$ and set $a=a_2-a_1$. Also fix a large integer $K$. Choose the starting points
\begin{equation}\label{1.1'}
\mu_j=
\begin{cases}
a_1+a(j-1)/K, &\text{if $1\le j\le n$,}\\
a_2+a(j-1)/K, &\text{if $n<j\le n+m$,}\\
\end{cases}
\end{equation}
where $N=n+m$. If we let $K\to\infty$ this will approach two starting points at $a_1$ and $a_2$. 
We assume that we have chosen our coordinate system so that
\begin{align}\label{1.1''}
b_1&=\max_{1\le j\le n}\nu_j<0,
\notag\\
b_2&=\min_{n<j\le n+m}\nu_j>0.
\end{align}
Let $\tilde{\mathcal{L}}_{K,m,n}(s,u,t,v;\nu)$ denote the kernel 
(\ref{1.1}) with this choice of initial and final points.
We will instead consider the kernel
\begin{equation}\label{1.2}
\mathcal{L}_{K,m,n}(s,u,t,v;\nu)=
\exp\left(\frac{u^2}{2(1-s)}-\frac{v^2}{2(1-t)}\right)\tilde{\mathcal{L}}_{K,m,n}(s,u,t,v;\nu),
\end{equation}
which gives the same correlation functions. Set
\begin{equation}\label{1.3}
q(s,u,t,v)=\exp\left(\frac{u^2}{2(1-s)}-\frac{v^2}{2(1-t)}\right)p_{t-s}(u,v).
\end{equation}
We want to take the limit of $\mathcal{L}_{K,m,n}(s,u,t,v;\nu)$ as $K\to\infty$. This will give the correlation
kernel for a process that we can interpret as $n$ Brownian particles starting at $a_1$ at time 0 and ending
at $\nu_1,\dots,\nu_m$ at time 1, together with $m$ Brownian particles starting at $a_2$ at time 0
and ending at $\nu_{n+1},\dots, \nu_{n+m}$ at time 1, conditioned not to collide during the 
time interval $(0,1)$. The asymptotic result is given in theorem \ref{thm1.2} and the limiting kernel, $\mathcal{L}_{m,n}(s,u,t,v;\nu)$
is defined by (\ref{1.14}).

\begin{remark}\label{rem1.1} \rm ({\bf Notation for contours of integration}) \rm In this paper we will need various contours
of integration. We collect the definitions here for reference. Given $c\in\mathbb{R}$  
let $\Gamma_c$ denote the contour $t\to c+it$, $t\in\mathbb{R}$, i.e. a vertical line through $c$. Let
$C_s$ denote a circle with radius $s$ and center at the origin, and let $D_1$ and $D_{-1}$ denote circles
with radii $<1$ and centers at $1$ and $-1$ respectively.
Also, let $\gamma_1$ denote a simple closed contour in the left half plane, $\re z<0$, that contains
$\nu_1,\dots,\nu_n$ in its interior, and let $\gamma_2$ denote
a simple closed contour in the right half plane, $\re z>0$, that contains
$\nu_{n+1},\dots,\nu_{n+m}$ in its interior. Furthermore, given $w$ with $\re w<0$, let 
$\gamma_{1,w}$ denote a simple closed contour in the left half plane that contains
$\nu_1,\dots,\nu_n$ and $w$ in its interior. Similarly, given $w$ with $\re w>0$, let $\gamma_{2,w}$ 
be a simple closed contour in the right half plane that contains
$\nu_{n+1},\dots,\nu_{n+m}$ and $w$ in its interior.
\it
\end{remark}

Let us first give the precise definition of the limiting kernel $\mathcal{L}_{m,n}(s,u,t,v;\nu)$
that we will obtain. Before we can do 
that we have to define some functions that we need in the definition.
Fix two numbers $d_1$ and $d_2$ that can depend on the parameters of the problem. 
They will be chosen appropriately when we take scaling limits, see
(\ref{1.18}). Set
\begin{align}\label{1.4}
\mathcal{B}_{v,t}(x)=\frac{d_2\sqrt{a}}{(2\pi i)^2\sqrt{1-t}}&\int_{\Gamma_c}\,dw\int_{\gamma_1}\,d\zeta
e^{\frac{t}{2(1-t)}w^2+a_1w-\frac{vw}{1-t}}
\notag\\
&\prod_{j=1}^n\frac{w/\nu_j-1}{\zeta/\nu_j-1}
\prod_{j=1}^m(1-\zeta/\nu_{n+j})\frac{e^{ax\zeta}}{\zeta-w},
\end{align}
where $c\ge 0$, and
\begin{equation}\label{1.5}
\beta_{v,t}(x)=\frac{d_2\sqrt{a}}{2\pi i\sqrt{1-t}}\int_{\Gamma_0}\,dw
e^{\frac{t}{2(1-t)}w^2+a_1w-\frac{vw}{1-t}+axw}
\prod_{j=1}^m(1-w/\nu_{n+j}).
\end{equation}
Also set,
\begin{align}\label{1.6}
\mathcal{C}_{u,s}(y)=\frac{d_1\sqrt{a}}{(2\pi i)^2\sqrt{1-s}}&\int_{\gamma_1}\,dz\int_{\gamma_2}\,d\omega
e^{-\frac{s}{2(1-s)}z^2-a_1z+\frac{uz}{1-s}}
\notag\\
&\times\prod_{j=1}^n\frac{\omega/\nu_j-1}{z/\nu_j-1}
\prod_{j=1}^m\frac{1}{1-\omega/\nu_{n+j}}\frac{e^{-ay\omega}}{\omega-z},
\end{align}
and
\begin{equation}\label{1.7}
M_0(x,y)=\frac{a}{(2\pi i)^2}\int_{\gamma_1}\,d\zeta\int_{\gamma_2}\,d\omega e^{ax\zeta-ay\omega}
\frac 1{\zeta-\omega}
\prod_{j=1}^n\frac{1-\omega/\nu_j}{1-\zeta/\nu_j}
\prod_{j=1}^m\frac{1-\zeta/\nu_{n+j}}{1-\omega/\nu_{n+j}}.
\end{equation}
Furthermore, we define
\begin{align}\label{1.12.3}
\hat{\mathcal{B}}_{v,t}(x)=\frac{d_2\sqrt{a}}{(2\pi i)^2\sqrt{1-t}}&\int_{\Gamma_c}\,dw\int_{\gamma_2}\,d\zeta
e^{\frac{t}{2(1-t)}w^2+a_2w-\frac{vw}{1-t}}
\notag\\
&\prod_{j=1}^m\frac{w/\nu_{n+j}-1}{\zeta/\nu_{n+j}-1}
\prod_{j=1}^n(\zeta/\nu_{j}-1)\frac{e^{-ax\zeta}}{\zeta-w},
\end{align}
where $c\le 0$, and
\begin{equation}\label{1.12.4}
\hat{\beta}_{v,t}(x)=\frac{d_2\sqrt{a}}{2\pi i\sqrt{1-t}}\int_{\Gamma_0}\,dw
e^{\frac{t}{2(1-t)}w^2+a_2w-\frac{vw}{1-t}-axw}
\prod_{j=1}^n(w/\nu_{j}-1).
\end{equation}
Also set,
\begin{align}\label{1.12.5}
\hat{\mathcal{C}}_{u,s}(y)=\frac{d_1\sqrt{a}}{(2\pi i)^2\sqrt{1-s}}&\int_{\gamma_2}\,dz\int_{\gamma_1}\,d\omega
e^{-\frac{s}{2(1-s)}z^2-a_2z+\frac{uz}{1-s}}
\notag\\
&\times\prod_{j=1}^m\frac{1- \omega/\nu_{n+j}}{1-z/\nu_{n+j}}
\prod_{j=1}^n\frac{1}{\omega/\nu_{j}-1}\frac{e^{ay\omega}}{z-\omega},
\end{align}
and
\begin{equation}\label{1.12.6}
\hat{M}_0(x,y)=\frac{a}{(2\pi i)^2}\int_{\gamma_2}\,d\zeta\int_{\gamma_1}\,d\omega e^{-ax\zeta+ay\omega}
\frac 1{\omega-\zeta}
\prod_{j=1}^m\frac{1-\omega/\nu_{n+j}}{1-\zeta/\nu_{n+j}}
\prod_{j=1}^n\frac{1-\zeta/\nu_{j}}{1-\omega/\nu_{j}}.
\end{equation}

To see that the expressions below are well defined we need the following lemma,
which will be proved in section \ref{sect4}.
\begin{lemma}\label{lem1.1}
The kernels $M_0(x,y)$ and $\hat{M}_0(x,y)$ are of finite rank on $L^2[1,\infty)$. Furthermore
$\det(I-M_0)_{L^2[1,\infty)}>0$ and $\det(I-\hat{M}_0)_{L^2[1,\infty)}>0$.
\end{lemma}

We can now define $L_{m,n}(s,u,t,v;\nu)$ by
\begin{align}\label{1.12}
&d_1d_2L_{m,n}(s,u,t,v;\nu)=
\frac{d_1d_2}{(2\pi i)^2\sqrt{(1-s)(1-t)}}\int_{\Gamma_0}\,dw
\int_{\gamma_1}\,dz\frac 1{w-z}
\notag\\
&\times e^{-\frac{s}{2(1-s)}z^2-a_1z+\frac{uz}{1-s}+
\frac{t}{2(1-t)}w^2+a_1w-\frac{vw}{1-t}}\prod_{j=1}^n\frac{1-w/\nu_j}{1-z/\nu_j}
\notag\\
&+\frac{\det(I-M_0+(\mathcal{B}_{v,t}+\beta_{v,t})\otimes\mathcal{C}_{u,s})_{L^2[1,\infty)}}
{\det(I-M_0)_{L^2[1,\infty)}}-1.
\end{align}
Similarly, we define $\hat{L}_{m,n}(s,u,t,v;\nu)$ by
\begin{align}\label{1.12.2}
&d_1d_2\hat{L}_{m,n}(s,u,t,v;\nu)=
\frac{d_1d_2}{(2\pi i)^2\sqrt{(1-s)(1-t)}}\int_{\Gamma_0}\,dw
\int_{\gamma_2}\,dz\frac 1{w-z}
\notag\\
&\times e^{-\frac{s}{2(1-s)}z^2-a_2z+\frac{uz}{1-s}+
\frac{t}{2(1-t)}w^2+a_2w-\frac{vw}{1-t}}\prod_{j=1}^m\frac{1-w/\nu_{n+j}}{1-z/\nu_{n+j}}
\notag\\
&+\frac{\det(I-\hat{M}_0+(\hat{\mathcal{B}}_{v,t}+\hat{\beta}_{v,t})\otimes
\hat{\mathcal{C}}_{u,s})_{L^2[1,\infty)}}
{\det(I-\hat{M}_0)_{L^2[1,\infty)}}-1.
\end{align}
Set
\begin{equation}\label{1.14}
\mathcal{L}_{m,n}(s,u,v,t;\nu)=L_{m,n}(s,u,t,v;\nu)+\hat{L}_{m,n}(s,u,t,v;\nu)-q(s,u,t,v).
\end{equation}
As discussed above this is the kernel for non-colliding Brownian motions with two starting points
and arbitrary endpoints.
Before stating our first main result we note that when $a_2=-a_1$  reflection in the time axis gives 
a simple relation  between $L_{n,m}$ and $\hat{L}_{m,n}$.
\begin{proposition}\label{prop1.1.1}
Assume that $a_2=-a_1$ and set $\hat{\nu}_{j}=-\nu_{n+m+1-j}$, $1\le j\le n+m$. Then,
\begin{equation}\label{1.12.7}
\hat{L}_{m,n}(s,u,t,v;\nu)=L_{n,m}(s,-u,t,-v;\hat{\nu}).
\end{equation}
\end{proposition}
\begin{proof}
Using the definitions we see that in this case $\hat{M}_0=M_0$, $\hat{\mathcal{B}}_{v,t}(x)=(-1)^n
\mathcal{B}_{-v,t}(x)$, $\hat{\beta}_{v,t}(x)=(-1)^n\beta_{-v,t}(x)$ and $\hat{\mathcal{C}}_{u,s}(y)
=(-1)^{n}\mathcal{C}_{-u,s}(y)$. If we use these identities in (\ref{1.12.2}) we see from (\ref{1.12})
that (\ref{1.12.7}) holds.
\end{proof}

We can now formulate our first main result.

\begin{theorem}\label{thm1.2}
We have the following pointwise limit,
\begin{equation}\label{1.15}
\mathcal{L}_{m,n}(s,u,v,t;\nu)=\lim_{K\to\infty}\mathcal{L}_{K,m,n}(s,u,t,v;\nu).
\end{equation}
\end{theorem}

The proof of this theorem is rather involved and will be given in section \ref{sect2}.

We want to use the formula for $\mathcal{L}_{m,n}(s,u,v,t;\nu)$ to investigate the scaling limit of 
the determinantal point processes defined by this kernel in a case
where we have a tacnode. A tacnode situation can be obtained when we have two groups of non-colliding
Brownian motions starting at two points and ending at two points. The starting and ending points are
chosen in such a way that in the limit $n,m\to\infty$ the global picture of the paths consists of two
tangent ellipses. 
With appropriate scaling near the point of tangency, the tacnode determinantal point process is obtained
in the limit.
We do not consider the most general possible geometry. 
Instead we consider a symmetric case.

Consider the case when $n=m$, $a_1=-a/2$, $a_2=a/2$, $\nu_j=b_1=-a/2$, $\nu_{n+j}=b_2=a/2$ for
$1\le j\le n$ with $a>0$. Call this choice $\nu^*$. Though we initially assumed that $\nu_1<\dots <\nu_N$,
we can let points coincide 
by taking a limit and using continuity. The determinantal point process with kernel $\mathcal{L}_{m,n}(s,u,v,t;\nu)$
is still well-defined.
Write
\begin{equation}
\mathcal{L}_n(s,u,t,v)=\mathcal{L}_{n,n}(s,u,t,v;\nu^*)
\notag
\end{equation}
and similarly for $L$ and $\hat{L}$. It follows from (\ref{1.12.7}) that
\begin{equation}\label{1.16}
\hat{L}_{n}(s,u,t,v)=L_{n}(s,-u,y,-v).
\end{equation}

Consider $n$ Brownian motions starting at 0 at time 0 and ending at
0 at time 1 and conditioned not to intersect. At time 1/2 the particles are distributed
as the eigenvalues of an $n\times n$  GUE matrix, and hence are approximately distributed as the semi circle
law for large n. Therefore particles approximately lie between $-\sqrt{n}$
and $\sqrt{n}$. Hence if we choose $a=2\sqrt{n}$ we expect that the two groups of Brownian 
motions just touch at time 1/2 in a region around the origin. The fluctuations of the largest eigenvalue
of a GUE matrix are of order $n^{-1/6}$. These considerations motivate the following choice of
scaling limit,
\begin{align}\label{1.17}
a&=2\sqrt{n}+\sigma n^{-1/6}
\notag\\
s&=\frac 12(1+\tau_1 n^{-1/3}),\quad t=\frac 12(1+\tau_2 n^{-1/3})
\notag\\
u&=\frac 12\xi_1n^{-1/6},\quad v=\frac 12\xi_2n^{-1/6}.
\end{align}
The parameter $\sigma$ measures how much the two groups of Brownian motions press against
each other, and the scaling in the time direction is such that we have the standard Brownian
space-time relation. The numbers $d_1,d_2$ in (\ref{1.12}) and (\ref{1.12.2}) will be chosen as
\begin{align}\label{1.18}
d_1&=\frac{n^{-1/12}}{\sqrt{2}}e^{\tau_1(\sigma+\xi_1)+\frac 23\tau_1^3}
\notag\\
d_2&=\frac{n^{-1/12}}{\sqrt{2}}e^{-\tau_2(\sigma+\xi_2)-\frac 23\tau_2^3}.
\end{align}
That $d_1d_2$ should contain the factor $n^{-1/6}$ can be seen from the volume element in (\ref{1.17}).
That $n^{-1/12}$ is the right choice for $d_1$ and $d_2$ individually is something that comes out of the asymptotic analysis. 
The choice of the other factors is somewhat arbitrary and different conjugations will lead to minor modifications
in the formulas below, e.g. we could make the formula (\ref{1.26}) more symmetric by a different choice. Set
\begin{align}\label{1.19}
&p(\tau_1,\xi_1,\tau_2,\xi_2)
\notag\\
&=\frac 1{\sqrt{4\pi(\tau_2-\tau_1)}}\exp\left(
-\frac {(\xi_1-\xi_2)^2}{4(\tau_2-\tau_1)}+\tau_1(\xi_1+\sigma)-\tau_2(\xi_2+\sigma)
-\frac 23\tau_2^3+\frac 23\tau_1^3\right),
\end{align}
and 
\begin{equation}\label{1.20}
\tilde{A}(\tau_1,\xi_1,\tau_2,\xi_2)=\int_0^\infty e^{\lambda(\tau_2-\tau_1)}\Ai(\xi_1+\lambda)
\Ai(\xi_2+\lambda)\,d\lambda.
\end{equation}
Also, let
\begin{equation}\label{1.21}
K_{\text{Ai}}(x,y)=\int_0^\infty\Ai(x+\lambda)\Ai(y+\lambda)\,d\lambda,
\end{equation}
be the {\it Airy kernel}. The {\it Tracy-Widom distribution} $F_2(s)$ is given by
\begin{equation}\label{1.22}
F_2(s)=\det(I-K_{\text{Ai}})_{L^2(s,\infty)}.
\end{equation}
Set
\begin{equation}\label{1.23}
B_{\xi,\tau}(x)=2^{1/6}\int_0^\infty e^{2^{1/3}\lambda\tau}\Ai(\xi+\tau^2+2^{1/3}\lambda)\Ai(x+\lambda)\,d\lambda
\end{equation}
and
\begin{equation}\label{1.24}
b_{\xi,\tau}(x)=2^{1/6}e^{-2\tau\xi+2^{1/3}\tau x}\Ai(-\xi+\tau^2+2^{1/3} x).
\end{equation}
Write $\tilde{\sigma}=2^{2/3}\sigma$ to simplify the notation. Define
\begin{align}\label{1.25}
&L_{\text{tac}}(\tau_1,\xi_1,\tau_2,\xi_2)=\tilde{A}(\tau_1,\xi_1+\tau_1^2+\sigma,\tau_2,\xi_2+\tau_2^2+\sigma)-1
\notag\\
&+\frac 1{F_2(\tilde{\sigma})}\det(I-K_{\text{Ai}}+(B_{\xi_2+\sigma,\tau_2}-b_{\xi_2+\sigma,\tau_2})
\otimes B_{\xi_1+\sigma,-\tau_1})_{L^2(\tilde{\sigma},\infty)}
\end{align}
and the {\it extended tacnode kernel},
\begin{align}\label{1.26}
\mathcal{L}_{\text{tac}}(\tau_1,\xi_1,\tau_2,\xi_2)&=L_{\text{tac}}(\tau_1,\xi_1,\tau_2,\xi_2)
+e^{2\tau_1\xi_1-2\tau_2\xi_2}L_{\text{tac}}(\tau_1,-\xi_1,\tau_2,-\xi_2)
\notag\\
&-p(\tau_1,\xi_1,\tau_2,\xi_2)\,1_{\tau_1<\tau_2}.
\end{align}
Note that the kernel in the Fredholm determinant in (\ref{1.25}) is a perturbation of the Airy kernel operator
with a rank one operator and hence is a trace class operator.
Our second main result is

\begin{theorem}\label{thm1.3} 
With  the scalings (\ref{1.17}) and (\ref{1.18}) we have the following pointwise limit
\begin{equation}\label{1.27}
\lim_{n\to\infty}d_1d_2\mathcal{L}_n(s,u,t,v)=\mathcal{L}_{\text{tac}}(\tau_1,\xi_1,\tau_2,\xi_2).
\end{equation}
\end{theorem}

We will prove the theorem in section \ref{sect3}. In the asymptotic analysis we use the fact that 
the objects that appear in 
$\mathcal{L}_n(s,u,t,v)$ can be expressed using Laguerre and Hermite polynomials, and we use 
their known asymptotics.

The formula (\ref{1.25}) can be written in different ways. Let us give another, somewhat more explicit,
version. It suffices to give
a formula for $L_{\text{tac}}$.

Let $R(x,y)$ be the resolvent operator for the restriction of the Airy kernel
to $[\tilde{\sigma},\infty)$, i.e. the kernel of the operator
\begin{equation}\label{1.28}
R=(I-K_{\text{Ai}})^{-1}K_{\text{Ai}}
\end{equation}
on $L^2[\tilde{\sigma},\infty)$.

\begin{proposition}\label{prop1.4}
We have the following formula,
\begin{align}\label{1.29}
&L_{\text{tac}}(\tau_1,\xi_1,\tau_2,\xi_2)=\tilde{A}(\tau_1,\xi_1+\tau_1^2+\sigma,\tau_2,\xi_2+\tau_2^2+\sigma)
\notag\\
&+2^{1/3}e^{2\sigma(\tau_1-\tau_2)}\int_{\tilde{\sigma}}^\infty\int_{\tilde{\sigma}}^\infty e^{2^{1/3}(\tau_2x-\tau_1y)}
\Ai(\xi_2+\tau_2^2-\sigma+2^{1/3}x)
\notag\\
&\times R(x,y)\Ai(\xi_1+\tau_1^2-\sigma+2^{1/3}y)\,dxdy
\notag\\
&-2^{1/3}e^{2\sigma(\tau_1-\tau_2)-2\tau_2\xi_2}\int_{\tilde{\sigma}}^\infty\int_{\tilde{\sigma}}^\infty e^{2^{1/3}(\tau_2x-\tau_1y)}
\Ai(-\xi_2-\sigma+\tau_2^2+2^{1/3}x)
\notag\\
&\times\Ai(x+y-\tilde{\sigma})\Ai(\xi_1+\tau_1^2-\sigma+2^{1/3}y)\,dxdy
\notag\\
&-2^{1/3}e^{2\sigma(\tau_1-\tau_2)-2\tau_2\xi_2}\int_{\tilde{\sigma}}^\infty\int_{\tilde{\sigma}}^\infty
\int_{\tilde{\sigma}}^\infty e^{2^{1/3}(\tau_2x-\tau_1z)}
\Ai(-\xi_2-\sigma+\tau_2^2+2^{1/3}x)
\notag\\
&\times R(x,y)\Ai(y+z-\tilde{\sigma})\Ai(\xi_1+\tau_1^2-\sigma+2^{1/3}z)\,dxdydz.
\end{align}
\end{proposition}
The proof is a computation and is given in section \ref{sect3}. An alternative tacnode kernel is given in
\cite{AFvM10}. This kernel also contains Airy-like objects just like (\ref{1.29}) but we have not been
able to show that they are equivalent, i.e. give rise to the same correlation functions.

\begin{remark}\label{rem4.3} \rm It follows from the estimates (\ref{2.17:4}), (\ref{2.17:5}) and (\ref{2.17:6}) that for $u,v$ in a
compact subset of $\mathbb{R}$ we have a uniform bound $|L_{K,m,n}(s,u,t,v;\nu)|\le C$.
A similar statement holds for $\hat{L}_{K,m,n}$. Together with the pointwise convergence result in theorem \ref{thm1.2} this can be
used to show that there actually is a point process with
determinantal correlation functions and correlation kernel $\mathcal{L}_{m,n}(s,u,t,v;\nu)$. This follows from a
theorem of Lenard, see \cite{So}, and the fact that $\mathcal{L}_{K,m,n}$ is the correlation kernel
for a determinantal point process. An analysis of the argument used to prove theorem \ref{thm1.3} gives a uniform bound of
$d_1d_2\mathcal{L}_n(s,u,t,v)$ for $\xi_1,\xi_2$ in a compact subset of $\mathbb{R}$, and this can similarly be used to
show that there is a determinantal point process with kernel $\mathcal{L}_{\text{tac}}$.\it
\end{remark}

\section{The kernel for two starting points and arbitrary endpoints}\label{sect2}
\subsection{Formula for the kernel}\label{sect2.1}

Let us first recall some results about Schur polynomials and Toeplitz determinants
that we need.
For a partition $\lambda=(\lambda_1,\dots,\lambda_N)$ and $x=(x_1,\dots,x_N)$ we set
\begin{equation}
a_\lambda(x)=\det(x_i^{\lambda_j})_{1\le i,j\le N}.
\notag
\end{equation}
Let $\delta=(N-1,N-2,\dots,1,0)$. Then
\begin{equation}
a_\delta(x)=\prod_{1\le i<j\le N}(x_i-x_j),
\notag
\end{equation}
is the Vandermonde determinant. The {\it Schur polynomial} $s_\lambda(x)$ is given by
\begin{equation}\label{2.1}
s_\lambda(x)=\frac{a_{\delta+\lambda}(x)}{a_\delta(x)},
\end{equation}
see e.g. \cite{Sa}. Let $e_r(x)$ be the $r$:th elementary symmetric polynomial with generating function
\begin{equation}\label{2.2}
\sum_{r\in\mathbb{Z}} e_r(x)\zeta^r=\prod_{j=1}^N(1+x_j\zeta).
\end{equation}
The Schur polynomial is also given by the Jacobi-Trudi identity
\begin{equation}\label{2.3}
s_{\mu'}(x)=\det(e_{\mu_i-i+j}(x))_{i,j=1}^K,
\end{equation}
where $\mu=(\mu_1,\dots,\mu_K)$ and $\mu'$ is the conjugate partition to $\mu$.
Write \linebreak
$<K^m>=(K,\dots,K,0,\dots,0)$ with $m$ parts equal to $K$. Note that 
\linebreak
$<m^K>'=<K^m>$.
Hence, by (\ref{2.3}),
\begin{equation}\label{2.4}
s_{<K^m>}(x)=\det(e_{m-i+j}(x))_{i,j=1}^K.
\end{equation}
The right side of (\ref{2.4}) is a Toeplitz determinant. Recall that if $f\in L^1(\mathbb{T})$, 
where $\mathbb{T}$ is the unit circle, the Toeplitz determinant $D_n[f(\zeta)]$ with symbol 
$f(\zeta)$, $\zeta\in\mathbb{T}$, is defined by
\begin{equation}
D_n[f(\zeta)]=\det(f_{j-i})_{1\le i,j\le n},
\end{equation}
where $f_k$ is the $k$:th Fourier coefficient of $f$. The generating function for the
Toeplitz determinant in the right side of (\ref{2.4}) is, by (\ref{2.2}),
\begin{align}
&\sum_{j\in\mathbb{Z}}e_{m+j}(x)\zeta^j=\zeta^{-m}\prod_{j=1}^N(1+x_j\zeta)
\notag\\
&=\prod_{j=n+1}^{N}x_j\prod_{j=1}^n(1+x_j\zeta)\prod_{j=n+1}^{N}(1+x_j^{-1}\zeta^{-1}),
\notag
\end{align}
where $N=n+m$ and we have assumed that $x_{n+1},\dots,x_N\neq 0$. Thus,
\begin{equation}\label{2.5}
s_{<K^m>}(x)=D_K\left[\prod_{j=n+1}^{N}x_j
\prod_{j=1}^n(1+x_j\zeta)\prod_{j=n+1}^{N}(1+x_j^{-1}\zeta^{-1})\right].
\end{equation}

Let us now turn to the problem of rewriting the kernel (\ref{1.1}) in a way that is suitable for
our problem. The initial steps are similar to those in \cite{Jo01}.
Let $p$ denote the column vector $p=(p_t(\mu_1,v)\dots p(\mu_n,v))^t$ and $(A|p)_k$ the
matrix $A$ with column $k$ replaced by $p$. By Cramer's rule and (\ref{1.1}) we then have
\begin{equation}\label{2.6}
\mathcal{L}(s,u,t,v;\mu,\nu)=-p_{t-s}(u,v)+\sum_{k=1}^Np_{1-s}(u,\nu_k)
\frac{\det (A|p)_k}{\det A}.
\end{equation}
Note that
\begin{equation}\label{2.7}
p_t(x,y)=\frac{e^{\frac{y^2}{2(1-t)}}}{i\sqrt{2\pi (1-t)}}\int_{\Gamma_c} e^{\frac 1{2(1-t)}
(w^2-2yw)}p_1(x,w)\,dw.
\end{equation}
Set
\begin{equation}
\nu_j^{(k)}=
\begin{cases}
\nu_j, &\text{if $j\neq k$,}\\
w, &\text{if $j=k$,}\\
\end{cases}
\end{equation}
Then, by (\ref{2.7}),
\begin{equation}\label{2.8}
\frac{\det (A|p)_k}{\det A}=
\frac{e^{\frac{v^2}{2(1-t)}}}{i\sqrt{2\pi (1-t)}}\int_{\Gamma_c} e^{\frac 1{2(1-t)}(w^2-2vw)}
\frac{\det(p_1(\mu_i,\nu_j^{(k)}))}{\det(p_1(\mu_i,\nu_j))}\,dw.
\end{equation}
Inserting this into (\ref{2.6}) and using the definition of $p_t(x,y)$ we find
\begin{align}
\mathcal{L}(s,u,t,v;\mu,\nu)&=-p_{t-s}(u,v)+\frac{e^{\frac{v^2}{2(1-t)}-\frac{u^2}{2(1-s)}}}
{2\pi i\sqrt{(1-s)(1-t)}}\int_{\Gamma_c} e^{\frac {w^2}{2(1-t)}-\frac{vw}{1-t}}
\notag\\
&\times\sum_{k=1}^N e^{-\frac{\nu_k^2}{2(1-s)}+\frac{u\nu_k}{1-s}}
\frac{\det(p_1(\mu_i,\nu_j^{(k)}))}{\det(p_1(\mu_i,\nu_j))}\,dw.
\notag
\end{align}
Hence, by (\ref{1.2}) and (\ref{1.3}) with the choice (\ref{1.1'}) of $\mu_j$,
\begin{align}\label{2.9}
\mathcal{L}_{K,m,n}(s,u,t,v;\nu)&=-q(s,u,t,v)+\frac{1}
{2\pi i\sqrt{(1-s)(1-t)}}\int_{\Gamma_c} e^{\frac {w^2}{2(1-t)}-\frac{vw}{1-t}}
\notag\\
&\times\sum_{k=1}^N e^{-\frac{\nu_k^2}{2(1-s)}+\frac{u\nu_k}{1-s}}
\frac{\det(p_1(\mu_i,\nu_j^{(k)}))}{\det(p_1(\mu_i,\nu_j))}\,dw.
\end{align}
Note that, by (\ref{1.1'}),
\begin{equation}\label{2.10}
\mu_{N+1-i}=a_1+\frac{a}K(\delta+<K^m>)_i.
\end{equation}
Using (\ref{2.10}) we see that
\begin{align}
\det(p_1(\mu_i,\nu_j))&=\frac 1{(2\pi)^{N/2}}\det(e^{-(\mu_j-\nu_i)^2/2})=
\frac 1{(2\pi)^{N/2}}\prod_{j=1}^N e^{-(\mu_j^2+\nu_j^2)/2}\det((e^{\nu_i})^{\mu_j})
\notag\\
&=\frac {(-1)^{N(N-1)/2}}{(2\pi)^{N/2}}\prod_{j=1}^N e^{-(\mu_j^2+\nu_j^2)/2+a_1\nu_j}
\det((e^{a\nu_i/K})^{(\delta+<K^m>)_j})
\notag\\
&=\frac {(-1)^{N(N-1)/2}}{(2\pi)^{N/2}}\prod_{j=1}^N e^{-(\mu_j^2+\nu_j^2)/2+a_1\nu_j}
a_{\delta+<K^m>}(x),
\notag
\end{align}
where $x_i=\exp(a\nu_i/K)$. If we set $x_i^{(k)}=\exp(a\nu_i^{(k)}/K)$, then by (\ref{2.1}),
\begin{equation}\label{2.11}
\frac{\det(p_1(\mu_i,\nu_j^{(k)}))}{\det(p_1(\mu_i,\nu_j))}=
\frac{s_{<K^m>}(x^{(k)})}{s_{<K^m>}(x)}\frac{a_\delta(x^{(k)})}{a_{\delta}(x)}
e^{(\nu_k^2-w^2)/2+a_1(w-\nu_k)}.
\end{equation}
Recall that $N=n+m$, and set
\begin{equation}\label{2.12}
g_K(\zeta)=\prod_{j=1}^n(1+e^{a\nu_j/K}\zeta)\prod_{j=n+1}^{n+m}(1+e^{-a\nu_j/K}\zeta^{-1}).
\end{equation}
Note that
\begin{equation}
\zeta^{-m}\prod_{j=1}^N (1+x_j^{(k)}\zeta)=\frac{1+e^{aw/K}\zeta}{1+e^{a\nu_k/K}\zeta}\zeta^{-m}
\prod_{j=1}^N (1+x_j\zeta).
\notag
\end{equation}
Hence, by (\ref{2.5}),
\begin{equation}\label{2.13}
\frac{s_{<K^m>}(x^{(k)})}{s_{<K^m>}(x)}=\frac{D_K\left[\frac{1+e^{aw/K}\zeta}{1+e^{a\nu_k/K}\zeta}
g_K(\zeta)\right]}{D_K[g_K(\zeta)]}.
\end{equation}
Also, note that by the definition of $a_\delta$,
\begin{equation}\label{2.14}
\frac{a_\delta(x^{(k)})}{a_{\delta}(x)}=\prod_{j\neq k}\frac{e^{aw/K}-e^{a\nu_j/K}}{e^{a\nu_k/K}-e^{a\nu_j/K}}.
\end{equation}
If we insert (\ref{2.13}) and (\ref{2.14}) into (\ref{2.11}) and use (\ref{2.9})
we obtain
\begin{align}\label{2.15}
&\mathcal{L}_{K,m,n}(s,u,t,v;\nu)=-q(s,u,t,v)+\frac 1{2\pi i\sqrt{(1-s)(1-t)}}
\int_{\Gamma_c}e^{\frac{t}{2(1-t)}w^2+a_1w-\frac{vw}{1-t}}
\notag\\
&\times
\sum_{k=1}^{n+m} e^{-\frac{s}{2(1-s)}\nu_k^2-a_1\nu_k+\frac{u\nu_k}{1-s}}
\frac{D_K\left[\frac{1+e^{aw/K}\zeta}{1+e^{a\nu_k/K}\zeta}
g_K(\zeta)\right]}{D_K[g_K(\zeta)]}
\prod_{j=1,j\neq k}^{n+m}\frac{e^{aw/K}-e^{a\nu_j/K}}{e^{a\nu_k/K}-e^{a\nu_j/K}}\,dw.
\end{align}
Note that
\begin{equation}\label{2.15.2}
D_K\left[\frac{1+e^{aw/K}\zeta}{1+e^{a\nu_k/K}\zeta}
g_K(\zeta)\right]=e^{a(w-\nu_k)}D_K\left[\frac{1+e^{-aw/K}\zeta}{1+e^{-a\nu_k/K}\zeta}
g_K(\zeta^{-1})\right]
\end{equation}
Set,
\begin{equation}\label{2.16}
F_K(w,k)=\frac{D_K\left[\frac{1+e^{aw/K}\zeta}{1+e^{a\nu_k/K}\zeta}
g_K(\zeta)\right]}{D_K[g_K(\zeta)]},
\end{equation}
for $1\le k\le n$, and
\begin{equation}\label{2.17}
F_K(w,k)=\frac{D_K\left[\frac{1+e^{-aw/K}\zeta}{1+e^{-a\nu_k/K}\zeta}
g_K(\zeta^{-1})\right]}{D_K[g_K(\zeta^{-1})]},
\end{equation}
for $n<k\le n+m$.
Define
\begin{align}\label{2.18}
&L_{K,m,n}(s,u,t,v;\nu)
=\frac 1{2\pi i\sqrt{(1-s)(1-t)}}
\int_{\Gamma_c}e^{\frac{t}{2(1-t)}w^2+a_1w-\frac{vw}{1-t}}
\notag\\
&\times
\sum_{k=1}^{n} e^{-\frac{s}{2(1-s)}\nu_k^2-a_1\nu_k+\frac{u\nu_k}{1-s}}F_K(w,k)
\prod_{j=1,j\neq k}^{n+m}\frac{e^{aw/K}-e^{a\nu_j/K}}{e^{a\nu_k/K}-e^{a\nu_j/K}}\,dw
\end{align}
and
\begin{align}\label{2.19}
&\hat{L}_{K,m,n}(s,u,t,v;\nu)
=\frac 1{2\pi i\sqrt{(1-s)(1-t)}}
\int_{\Gamma_c}e^{\frac{t}{2(1-t)}w^2+a_2w-\frac{vw}{1-t}}
\notag\\
&\times
\sum_{k=n+1}^{n+m} e^{-\frac{s}{2(1-s)}\nu_k^2-a_2\nu_k+\frac{u\nu_k}{1-s}}F_K(w,k)
\prod_{j=1,j\neq k}^{n+m}\frac{e^{aw/K}-e^{a\nu_j/K}}{e^{a\nu_k/K}-e^{a\nu_j/K}}\,dw.
\end{align}
Then,
\begin{equation}\label{2.20}
\mathcal{L}_{K,m,n}(s,u,t,v;\nu)=-q(s,u,t,v)+L_{K,m,n}(s,u,t,v;\nu)+\hat{L}_{K,m,n}(s,u,t,v;\nu).
\end{equation}
This representation is useful for the analysis of the limit $K\to\infty$.

\subsection{Proof of theorem \ref{thm1.2}}\label{sect2.2}

We want to take the limit $K\to\infty$ in (\ref{2.18}) and (\ref{2.19}). Before we can do that we
rewrite the Toeplitz determinants in (\ref{2.16}) and (\ref{2.17}) 
using the Geronimo-Case/Borodin-Okounkov (GCBO)
identity, \cite{GC79}, \cite{BO99}. Consider first the kernel $L_{K,m,n}(s,u,t,v;\nu)$.
To simplify the notation we write
\begin{align}
\alpha&=e^{aw/K},\,\quad \beta=e^{az/K},
\notag\\
\gamma_j&=e^{a\nu_j/K},\,\, 1\le j\le n, \quad \delta_j=e^{-a\nu_{n+j}/K},\,\, 1\le j\le m.
\notag
\end{align}
We assume that $\re z<0$ and $\re w<0$ so that all numbers have absolute value less than 1.
This means that we assume that $c<0$ in $\Gamma_c$ in (\ref{2.18}).
The function $g_K$ in (\ref{2.12}) can now be written
\begin{equation}\label{2.21}
g_K(\zeta)=\prod_{j=1}^n(1+\gamma_j\zeta)\prod_{j=1}^m(1+\delta_j\zeta^{-1}).
\end{equation}
Note that $\beta=\gamma_k$ if $z=\nu_k$. 

We want to express
\begin{equation}
D_K\left[\frac{1+\alpha\zeta}{1+\beta\zeta}g_K(\zeta)\right]
\notag
\end{equation}
using the GCBO identity. We use the formulation given in \cite{BW00}. Let
\begin{equation}
\phi(\zeta)=\frac{1+\alpha\zeta}{1+\beta\zeta}g_K(\zeta)=\phi_+(\zeta)\phi_-(\zeta),
\notag
\end{equation}
where
\begin{align}
\phi_+(\zeta)&=\frac{1+\alpha\zeta}{1+\beta\zeta}\prod_{j=1}^n(1+\gamma_j\zeta),
\notag\\
\phi_-(\zeta)&=\prod_{j=1}^m(1+\delta_j\zeta^{-1}).
\notag
\end{align}
Note that $\phi_+$ is analytic and non-zero in $|\zeta|<1$, and $\phi_-$ is analytic and non-zero
in $|\zeta|>1$. Also, $\phi_+(0)=\phi_-(\infty)=1$. Let
\begin{equation}\label{2.21'}
\mathcal{K}(r,s)=\sum_{\ell=1}^\infty \left(\frac{\phi_-}{\phi_+}\right)_{r+\ell}
 \left(\frac{\phi_+}{\phi_-}\right)_{-\ell-s}
\end{equation}
and
\begin{equation}
Z=\sum_{j=1}^\infty j(\log\phi)_j(\log\phi)_{-j}.
\notag
\end{equation}
Then the GCBO identity states that
\begin{equation}\label{2.22}
D_K[\phi(\zeta)]=e^Z\det(I-\mathcal{K})_{\bar{\ell}^2(K)},
\end{equation}
where we have introduced the notation
$\bar{\ell}^2(K)=\ell^2(\{K,K+1,\dots\})$.

A computation gives
\begin{equation}\label{2.23}
e^Z=\prod_{j=1}^m\frac{1-\beta\delta_j}{1-\alpha\delta_j}\prod_{i=1}^n\prod_{j=1}^m(1-\gamma_i\delta_j)^{-1}.
\end{equation}
We also get
\begin{equation}\label{2.24}
\left(\frac{\phi_-}{\phi_+}\right)_{r+\ell}=\frac 1{2\pi i}\int_{C_{s_1}}\frac{1+\beta\zeta}{1+\alpha\zeta}
\prod_{j=1}^m(1+\delta_j\zeta^{-1})\prod_{j=1}^n(1+\gamma_j\zeta)^{-1}\zeta^{-(r+\ell)}\frac{d\zeta}{\zeta},
\end{equation}
where $C_s$ denotes a circle of radius $s$ around 0. We have to assume that $|\alpha s_1|<1$, $\gamma_js_1<1$,
$1\le j\le n$, i.e.
\begin{equation}\label{2.25}
s_1<e^{-a\re w/K}, e^{-a\nu_j/K},\,\,1\le j\le n.
\end{equation}
Furthermore,
\begin{equation}\label{2.26}
\left(\frac{\phi_+}{\phi_-}\right)_{-s-\ell}=\frac 1{2\pi i}\int_{C_{s_2}}\frac{1+\alpha\omega}{1+\beta\omega}
\prod_{j=1}^m(1+\delta_j\omega^{-1})^{-1}\prod_{j=1}^n(1+\gamma_j\omega)\omega^{s+\ell}\frac{d\omega}{\omega},
\end{equation}
where we require $|\beta s_2|<1$ and $|\delta s_2^{-1}|<1$, i.e.
\begin{equation}\label{2.27}
e^{-a\nu_{n+j}/K}<s_2<e^{-a\re z/K}, \quad 1\le j\le m.
\end{equation}
Inserting (\ref{2.24}) and (\ref{2.26}) into (\ref{2.21'}) we obtain (after changing $\zeta$ to $-\zeta$ and
$\omega$ to $-\omega$),
\begin{equation}\label{2.28}
\mathcal{K}(r,s)=\frac{(-1)^{r-s}}{(2\pi i)^2}\int_{C_{s_1}}\,d\zeta\int_{C_{s_2}}\,d\omega
\frac{(1-\beta\zeta)(1-\alpha\omega)}{(1-\alpha\zeta)(1-\beta\omega)}\frac{\omega^s}{\zeta^{r+1}}
\frac 1{\zeta-\omega}\frac{\tilde{H}_{m,n}(\zeta)}{\tilde{H}_{m,n}(\omega)},
\end{equation}
provided that also
\begin{equation}\label{2.29}
s_1>s_2.
\end{equation}
Here,
\begin{equation}\label{2.29'}
\tilde{H}_{m,n}(\zeta)=\prod_{j=1}^m(1-\delta_j\zeta^{-1})\prod_{j=1}^n(1-\gamma_j\zeta)^{-1}.
\end{equation}

Hence, by (\ref{2.22}) and (\ref{2.23}),
\begin{equation}\label{2.30}
D_K\left[\frac{1+\alpha\zeta}{1+\beta\zeta}g_K(\zeta)\right]
=\prod_{j=1}^m\frac{1-\beta\delta_j}{1-\alpha\delta_j}\prod_{i=1}^n\prod_{j=1}^m(1-\gamma_i\delta_j)^{-1}
\det(I-\mathcal{K})_{\bar{\ell}^2(K)}
\end{equation}
with $\mathcal{K}$ given by (\ref{2.28}), and where $s_1,s_2$ satisfy (\ref{2.25}), (\ref{2.27}) and (\ref{2.29}).
We can take $w=z$ so that $\alpha=\beta$. This gives
\begin{equation}\label{2.31}
D_K\left[g_K(\zeta)\right]
=\prod_{i=1}^n\prod_{j=1}^m(1-\gamma_i\delta_j)^{-1}
\det(I-\mathcal{K}_0)_{\bar{\ell}^2(K)},
\end{equation}
where
\begin{equation}\label{2.32}
\mathcal{K}_0(r,s)=\frac{(-1)^{r-s}}{(2\pi i)^2}\int_{C_{s_1'}}\,d\zeta\int_{C_{s_2'}}\,d\omega
\frac{\omega^s}{\zeta^{r+1}}
\frac 1{\zeta-\omega}\frac{\tilde{H}_{m,n}(\zeta)}{\tilde{H}_{m,n}(\omega)},
\end{equation}
From (\ref{2.25}), (\ref{2.27}) and (\ref{2.29}) we can summarize the restrictions on $s_1, s_2$,
\begin{align}\label{2.33}
s_1&<\min(e^{-a\re w/K},e^{-a\nu_i/K}),
\notag\\
e^{-a\nu_{n+j}/K}<s_2&<\min(s_1,e^{-a\re z/K}),
\end{align}
$1\le i\le n$, $1\le j\le m$.
For $s_1',s_2'$ there are no conditions related to $z$ and $w$ so we get instead
\begin{equation}\label{2.34}
e^{-a\nu_{n+i}/K}<s_2'<s_1'<e^{-a\nu_j/K},
\end{equation}
$1\le i\le m$, $1\le j\le n$.

The factor $(-1)^{r-s}$ in (\ref{2.28}) and (\ref{2.32}) can be removed without changing the Fredholm 
determinant. The change of variables $\omega\to1/\omega$ then gives the following kernel, which we
denote by $\mathcal{K}^{(z,w)}(r,s;K)$ to indicate the dependence on $z,w$ and $K$,
\begin{align}\label{2.35}
&\mathcal{K}^{(z,w)}(r,s;K)
\notag\\
&=\frac{1}{(2\pi i)^2}\int_{C_{s_1}}\,d\zeta\int_{C_{s_3}}\,d\omega
\frac{(1-\beta\zeta)(1-\alpha\omega^{-1})}{(1-\alpha\zeta)(1-\beta\omega^{-1})}\frac{1}{\omega^{s}\zeta^{r+1}}
\frac 1{\zeta\omega-1}\frac{\tilde{H}_{m,n}(\zeta)}{\tilde{H}_{m,n}(\omega^{-1})},
\end{align}
where $s_3=1/s_2$. We get the following condition on $s_3$ from (\ref{2.33}),
\begin{equation}
\max(\frac 1{s_1},e^{a\re z/K})<s_3<e^{a\nu_{n+j}/K},
\notag
\end{equation}
$1\le j\le m$. We can take $s_1,s_3>1$ in (\ref{2.35}) and notice that we are interested in $r,s\ge K$,
where $K$ is very large, actually tending to infinity. 
Looking at the integrand in (\ref{2.35}) we see that in $\zeta$ we have poles at $\zeta=0$, $\zeta=1/\omega$,
$\zeta=1/\alpha\approx 1-aw/K$ and $\zeta=1/\gamma_j\approx1-a\nu_j/K$, $1\le j\le n$. In the $\omega$ variable
we have poles at $\omega=0$, $\omega=1/\zeta$, $\omega=\beta\approx 1+az/K$ and $\omega=1/\delta_j\approx 1+a\nu_{n+j}/K$, 
$1\le j\le m$, compare (\ref{2.33}). Hence, we can, for large $r,s$, deform $C_{s_1}$ and
$C_{s_3}$ through infinity without hitting any poles to curves given by
\begin{equation}\label{2.36}
\zeta=1-\frac{a\zeta'}{K}, \,\,\zeta'\in \gamma_{1,w},\quad \omega=1+\frac{a\omega'}K, \,\,\omega'\in\gamma_2,
\end{equation}
where $\gamma_{1,w}$ and $\gamma_2$ are as defined in Remark \ref{rem1.1}. 
We obtain
\begin{align}\label{2.37}
&\mathcal{K}^{(z,w)}(r,s;K)=\frac {-a}{(2\pi i)^2K}\int_{\gamma_{1,w}}\,d\zeta'\int_{\gamma_2}\,d\omega'
\frac{(1-\beta(1-a\zeta'/K))(1-\alpha(1+a\omega'/K)^{-1})}
{(1-\alpha(1-a\zeta'/K))(1-\beta(1+a\omega'/K)^{-1})}
\notag\\
&\times \frac 1{(1+a\omega'/K)^{s}(1-a\zeta'/K)^{r+1}}
\frac 1{\frac Ka[(1-a\zeta'/K)(1+a\omega'/K)-1]}
\frac{\tilde{H}_{m,n}(1-a\zeta'/K)}{\tilde{H}_{m,n}((1+a\omega'/K)^{-1})}.
\end{align}
Define
\begin{equation}\label{2.38}
\tilde{\mathcal{K}}^{(z,w)}(x,y;K)=K\mathcal{K}^{(z,w)}([xK],[yK];K).
\end{equation}
 We can expand the determinant $\det(I-\mathcal{K}^{(z,w)})_{\bar{\ell}(K)}$ in a Fredholm expansion and use (\ref{2.38})
to obtain
\begin{align}\label{2.38'}
\det(I-\mathcal{K}^{(z,w)})_{\bar{\ell}^2(K)}&=\sum_{m=0}^\infty\frac{(-1)^m}{m!}\sum_{r_1,\dots r_m=K}^\infty
\det(\mathcal{K}^{(z,w)}(r_i,r_j;K))_{m\times m}
\notag\\
&=\sum_{m=0}^\infty\frac{(-1)^m}{m!}\int_{[1,\infty)^m}\det(\tilde{\mathcal{K}}^{(z,w)}(x_i,x_j;K))_{m\times m}
\,d^mx.
\end{align}
In (\ref{2.38}) we can take the limit $K\to\infty$. Note that
\begin{align}
\tilde{H}_{m,n}(1-a\zeta'/K)&=\prod_{j=1}^m (1-e^{-a\nu_{n+j}/K}(1-a\zeta'/K)^{-1})\prod_{j=1}^n
\left(1-e^{a\nu_j/K}(1-a\zeta'/K)\right)^{-1}
\notag\\
&\sim\left(\frac{a}{K}\right)^{m-n}\left(\prod_{j=1}^m\nu_{n+j}\prod_{j=1}^n\nu_j^{-1}\right) H_{m,n}(\zeta),
\notag
\end{align}
as $K\to\infty$, where
\begin{equation}\label{2.41}
H_{m,n}(\zeta)=\prod_{j=1}^m(1-\zeta/\nu_{n+j})\prod_{j=1}^n(\zeta/\nu_j-1)^{-1}.
\end{equation}
Also,
\begin{align}
\tilde{H}_{m,n}((1-a\omega'/K)^{-1})&=\prod_{j=1}^m (1-e^{-a\nu_{n+j}/K}(1+a\omega'/K))\prod_{j=1}^n
\left(1-e^{a\nu_j/K}(1+a\omega'/K)^{-1}\right)^{-1}
\notag\\
&\sim\left(\frac{a}{K}\right)^{m-n}\left(\prod_{j=1}^m\nu_{n+j}\prod_{j=1}^n\nu_j^{-1}\right) H_{m,n}(\omega),
\notag
\end{align}
as $K\to\infty$. The other parts of the integrand in (\ref{2.37}) are easy to analyze as $K\to\infty$ and it is not hard to see
that for any fixed $x,y,z,w$ all the limits hold uniformly for $\zeta'\in\gamma_{1,w}$ and $\omega'\in\gamma_2$. In this way we 
see that pointwise in $x,y,z,w$,
\begin{equation}\label{2.39}
\lim_{K\to\infty}\tilde{\mathcal{K}}^{(z,w)}(x,y;K)=M^{(z,w)}(x,y),
\end{equation}
where
\begin{equation}\label{2.40}
M^{(z,w)}(x,y)=\frac{a}{(2\pi i)^2}\int_{\gamma_{1,w}}\,d\zeta\int_{\gamma_2}\,d\omega
\frac{(z-\zeta)(w-\omega)}{(w-\zeta)(z-\omega)}\frac{e^{ax\zeta-ay\omega}}{\zeta-\omega}
\frac{H_{m,n}(\zeta)}{H_{m,n}(\omega)},
\end{equation}
for $\re z<0$, $\re w<0$, $x,y\ge 1$. 
Note that
\begin{equation}
(1-a\zeta'/K)^{-[xK]}=\exp(-[xK]\log(1-a\zeta'/K))=\exp([xK]\sum_{k=1}^\infty(a\zeta'/K)^m/m)
\notag
\end{equation}
for $\zeta'\in\gamma_{1,w}$ and large enough $K$. Since $\gamma_{1,w}$ lies in the open left half plane
there is an $\epsilon>0$ such that $\re\zeta'\le -\epsilon<0$ for all $\zeta'\in\gamma_{1,w}$. Hence,
\begin{equation}
\left|(1-a\zeta'/K)^{-[xK]}\right|\le e^{-a\epsilon x/2}
\notag
\end{equation}
for all $\zeta'\in\gamma_{1,w}$ and large $K$. In this way we see from (\ref{2.37}) that there is a constant $C$ 
independent of $K$ such that
\begin{equation}\label{2.42}
\left|\tilde{\mathcal{K}}^{(z,w)}(x,y;K)\right|\le C e^{-a\epsilon(x+y)/2}
\end{equation}
for all $x,y\ge 1$, $K$ sufficiently large.

It follows from (\ref{2.38'}), (\ref{2.39}), (\ref{2.42}), the Hadamard inequality and the dominated convergence theorem
that
\begin{align}\label{2.43}
&\lim_{K\to\infty}\det(I-\mathcal{K}^{(z,w)})_{\bar{\ell}^2(K)}\notag\\
&=\sum_{m=0}^\infty\frac{(-1)^m}{m!}
\int_{[1,\infty)^m}\det(M^{(z,w)}(x_i,x_j))_{m\times m}\,d^mx,
\end{align}
where $\mathcal{K}^{(z,w)}$ is given by (\ref{2.35}) and $M^{(z,w)}$ by (\ref{2.40}). By lemma \ref{lem4.1} (a) $M^{(z,w)}$
is a finite rank operator and hence the Fredholm determinant exists and equals its Fredholm expansion.
We obtain
\begin{equation}\label{2.49}
\lim_{K\to\infty}\det(I-\mathcal{K}^{(z,w)})_{\bar{\ell}^2(K)}=
\det(I-M^{(z,w)})_{L^2[1,\infty)}.
\end{equation}
Set $M_0=M^{(z,z)}$ so that $M_0$ is given by (\ref{1.7}). It then follows from (\ref{2.49})
and (\ref{2.32}) that 
\begin{equation}\label{2.50}
\lim_{K\to\infty}\det(I-\mathcal{K}_0)_{\bar{\ell}^2(K)}=
\det(I-M_0)_{L^2[1,\infty)},
\end{equation}
and we know from lemma \ref{lem4.1} (b) that the right hand side is positive.

If we combine (\ref{2.16}), (\ref{2.30}) and (\ref{2.31}) we see that
\begin{equation}
F_K(w,k)=\prod_{j=1}^m\frac{1-\gamma_k\delta_j}{1-\alpha\delta_j}\frac{\det(I-\mathcal{K}^{(\nu_k,w)})_
{\bar{\ell}^2(K)}}{\det(I-\mathcal{K}_0)_{\bar{\ell}^2(K)}}.
\notag
\end{equation}
Since
\begin{equation}
\lim_{K\to\infty}\prod_{j=1}^m\frac{1-\gamma_k\delta_j}{1-\alpha\delta_j}
=\prod_{j=1}^m\frac{\nu_{n+j}-\nu_k}{\nu_{n+j}-w},
\notag
\end{equation}
we see from (\ref{2.49}) and (\ref{2.50}) that, pointwise in $w$, $\re w<0$,
\begin{equation}\label{2.51}
\lim_{K\to\infty}F_K(w,k)=\prod_{j=1}^m\frac{\nu_{n+j}-\nu_k}{\nu_{n+j}-w}
\frac{\det(I-M^{(\nu_k,w)})_{L^2[1,\infty)}}{\det(I-M_0)_{L^2[1,\infty)}}
\end{equation}
for $1\le k\le n$.

We can now use (\ref{2.51}), the estimates (\ref{2.17:4}), (\ref{2.17:6}), which will be proved
in section \ref{sect4}, and the dominated 
convergence theorem to conclude that, pointwise in $s,u,t,v$,
\begin{align}\label{2.52}
&\lim_{K\to\infty}L_{m,n,K}(s,u,v,t;\nu)=\frac 1{2\pi i\sqrt{(1-s)(1-t)}}\int_{\Gamma_c}\,dw
e^{\frac{t}{2(1-t)}w^2+a_1w-\frac{vw}{1-t}}
\notag\\
&\times\sum_{k=1}^n e^{-\frac s{2(1-s)}\nu_k^2-a_1\nu_k+\frac{u\nu_k}{1-s}}
\prod_{j=1, j\neq k}^n\frac{w-\nu_j}{\nu_k-\nu_j}
\frac{\det(I-M^{(\nu_k,w)})_{L^2[1,\infty)}}{\det(I-M_0)_{L^2[1,\infty)}}
\end{align}
for $c<0$.

The argument used above to analyze $L_{m,n,K}(s,u,v,t;\nu)$ can also be applied to 
$\hat{L}_{m,n,K}(s,u,v,t;\nu)$. If we look at (\ref{2.17}) we see that we can perform exactly
the same argument with the changes $n\leftrightarrow m$, $\nu_1,\dots,\nu_n\to
-\nu_{n+1},\dots, -\nu_{n+m}$ and $\nu_{n+1},\dots,\nu_{n+m}\to-\nu_1,\dots,-\nu_n$, as well as $w\to -w$
and $z\to -z$. Define,
\begin{equation}\label{2.53}
\hat{M}^{(z,w)}(x,y)=\frac{a}{(2\pi i)^2}\int_{\gamma_{1}}\,d\zeta\int_{\gamma_{2,w}}\,d\omega
\frac{(z-\zeta)(w-\omega)}{(w-\zeta)(z-\omega)}\frac{e^{-ax\zeta+ay\omega}}{\omega-\zeta}
\frac{H_{m,n}(\omega)}{H_{m,n}(\zeta)},
\end{equation}
for $x,y\ge 1$, $\re z>0$, $\re w>0$. Also, set $\hat{M}_0=\hat{M}^{(z,z)}$ so that $\hat{M}_0$ is given by (\ref{1.12.6}).
We get,
\begin{equation}\label{2.54}
\lim_{K\to\infty}F_K(w,k)=\prod_{j=1}^n\frac{\nu_{j}-\nu_k}{\nu_{j}-w}
\frac{\det(I-\hat{M}^{(\nu_k,w)})_{L^2[1,\infty)}}{\det(I-\hat{M}_0)_{L^2[1,\infty)}}
\end{equation}
pointwise in $w$ for $\re w>0$, $n<k\le n+m$. Similarly to above we obtain, pointwise in $s,u,t,v$,
\begin{align}\label{2.55}
&\lim_{K\to\infty}\hat{L}_{m,n,K}(s,u,v,t;\nu)=\frac 1{2\pi i\sqrt{(1-s)(1-t)}}\int_{\Gamma_c}\,dw
e^{\frac{t}{2(1-t)}w^2+a_2w-\frac{vw}{1-t}}
\notag\\
&\times\sum_{k=n+1}^{n+m} e^{-\frac s{2(1-s)}\nu_k^2-a_2\nu_k+\frac{u\nu_k}{1-s}}
\prod_{j=n+1, j\neq k}^{n+m}\frac{w-\nu_j}{\nu_k-\nu_j}
\frac{\det(I-\hat{M}^{(\nu_k,w)})_{L^2[1,\infty)}}{\det(I-\hat{M}_0)_{L^2[1,\infty)}}
\end{align}
for $c>0$. It follows from the residue theorem that the right hand side of (\ref{2.52}) can be written as
\begin{align}\label{2.56}
&L_{m,n}(s,u,v,t;\nu)=\frac 1{(2\pi i)^2\sqrt{(1-s)(1-t)}}\int_{\gamma_1}\,dz\int_{\Gamma_{c_1}}\,dw
e^{\frac{t}{2(1-t)}w^2+a_1w-\frac{vw}{1-t}}
\notag\\
&e^{-\frac s{2(1-s)}z^2-a_1z+\frac{uz}{1-s}}\frac 1{w-z}
\prod_{j=1}^n\frac{w-\nu_j}{z-\nu_j}
\frac{\det(I-M^{(z,w)})_{L^2[1,\infty)}}{\det(I-M_0)_{L^2[1,\infty)}},
\end{align}
and similarly for the right hand side of (\ref{2.55}) we obtain
\begin{align}\label{2.57}
&\hat{L}_{m,n}(s,u,v,t;\nu)=\frac 1{(2\pi i)^2\sqrt{(1-s)(1-t)}}\int_{\gamma_2}\,dz\int_{\Gamma_{c_2}}\,dw
e^{\frac{t}{2(1-t)}w^2+a_2w-\frac{vw}{1-t}}
\notag\\
&e^{-\frac s{2(1-s)}z^2-a_2z+\frac{uz}{1-s}}\frac 1{w-z}
\prod_{j=n+1}^{n+m}\frac{w-\nu_j}{z-\nu_j}
\frac{\det(I-\hat{M}^{(z,w)})_{L^2[1,\infty)}}{\det(I-\hat{M}_0)_{L^2[1,\infty)}},
\end{align}
Here $c_1<0$ is chosen so that $\Gamma_{c_1}$ lies to the right of $\gamma_1$ and $c_2>0$ is chosen so that
$\Gamma_{c_2}$ lies to the left of $\gamma_2$. 

To prove theorem \ref{thm1.2} it remains to show that these expressions agree with those given in 
(\ref{1.12}) and (\ref{1.12.2}) respectively. Since
\begin{equation}
\frac{(z-\zeta)(w-\omega)}{(w-\zeta)(z-\omega)}=1-\frac{(w-z)(\zeta-\omega)}{(\zeta-w)(\omega-z)}
\notag
\end{equation}
we see that
\begin{equation}\label{2.60}
M^{(z,w)}=M_0-(w-z)b_1^w\otimes b_2^z
\end{equation}
as operators on $L^2[1,\infty)$, where
\begin{equation}\label{2.58}
b_1^w(x)=\frac{\sqrt{a}}{2\pi i}\int_{\gamma_{1,w}}H_{m,n}(\zeta)\frac{e^{ax\zeta}}{\zeta-w}\,d\zeta
\end{equation}
and
\begin{equation}\label{2.59}
b_2^z(y)=\frac{\sqrt{a}}{2\pi i}\int_{\gamma_2}H_{m,n}(\omega)^{-1}\frac{e^{-ay\omega}}{\omega-z}\,d\omega.
\end{equation}
Define
\begin{equation}\label{2.61}
F_{u,s}(z)=\frac {d_1}{2\pi i\sqrt{1-s}}e^{-\frac s{2(1-s)}z^2-a_1z+\frac{uz}{1-s}}
\prod_{j=1}^n\left(\frac z{\nu_j}-1\right)^{-1}
\end{equation}
and
\begin{equation}\label{2.62}
G_{v,t}(w)=\frac {d_2}{2\pi i\sqrt{1-t}}e^{\frac{t}{2(1-t)}w^2+a_1w-\frac{vw}{1-t}}\prod_{j=1}^n\left(\frac w{\nu_j}
-1\right).
\end{equation}
From (\ref{2.56}) we now obtain
\begin{align}\label{2.63}
d_1d_2 L_{m,n}(s,u,t,v;\nu)&=\int_{\gamma_1}\,dz\int_{\Gamma_{c_1}}\,dw F_{u,s}(z)G_{v,t}(w)\frac 1{w-z}
\notag\\
&\times\frac {\det(I-M_0+(w-z)b_1^w\otimes b_2^z)_{L^2[1,\infty)}}{\det(I-M_0)_{L^2[1,\infty)}}.
\end{align}
Next we use the following lemma that will be proved in section \ref{sect4}.

\begin{lemma}\label{lem2.2}
The following identity holds
\begin{align}\label{2.64}
&\int_{\gamma_1}\,dz\int_{\Gamma_{c_1}}\,dw F_{u,s}(z)G_{v,t}(w)\frac 1{w-z}
\det(I-M_0+(w-z)b_1^w\otimes b_2^z)_{L^2[1,\infty)}
\notag\\
&=\left[\int_{\gamma_1}\,dz\int_{\Gamma_{c_1}}\,dw F_{u,s}(z)G_{v,t}(w)\frac 1{w-z}-1\right]
\det(I-M_0)_{L^2[1,\infty)}
\notag\\
&+\det\left(I-M_0+\left(\int_{\Gamma_{c_1}}G_{v,t}(w)b_1^w\,dw\right)\otimes
\left(\int_{\gamma_1}F_{u,s}(z)b_2^z\,dz\right)\right)_{L^2[1,\infty)}.
\end{align}
\end{lemma}

Now, 
\begin{align}\label{2.65}
\int_{\gamma_1}F_{u,s}(z)b_2^z\,dz&=\frac {d_1\sqrt{a}}{(2\pi i)^2\sqrt{1-s}}\int_{\gamma_1}\,dz
\int_{\gamma_2}\,d\omega e^{-\frac s{2(1-s)}z^2-a_1z+\frac{uz}{1-s}}
\notag\\
&\times \prod_{j=1}^n\left(\frac z{\nu_j}-1\right)^{-1}H_{m,n}(\omega)^{-1}\frac{e^{-ay\omega}}{\omega-z}=\mathcal{C}_{u,s}(y),
\end{align}
where $\mathcal{C}_{u,s}(y)$ is given by (\ref{1.6}). Also,
\begin{align}
\int_{\Gamma_{c_1}}G_{v,t}(w)b_1^w\,dw&=\frac {d_2\sqrt{a}}{(2\pi i)^2\sqrt{1-t}}
\int_{\Gamma_{c_1}}dw\int_{\gamma_{1,w}}\,d\zeta e^{\frac t{2(1-t)}w^2+a_1w-\frac{vw}{1-t}}
\notag\\
&\times\prod_{j=1}^n\frac{w/\nu_j-1}{\zeta/\nu_j-1}\prod_{j=1}^m(1-\zeta/\nu_{j+n})\frac {e^{ax\zeta}}{\zeta-w}.
\notag
\end{align}
We can move $\gamma_{1,w}$ to a contour $\gamma'_{1,w}$ in the left half plane that contains
$\nu_1,\dots,\nu_n$ but does not contain $w$. This gives,
\begin{align}\label{2.66}
&\int_{\Gamma_{c_1}}G_{v,t}(w)b_1^w\,dw=\frac {d_2\sqrt{a}}{(2\pi i)^2\sqrt{1-t}}
\int_{\Gamma_{c_1}}\,dw\int_{\gamma'_{1,w}}\,d\zeta e^{\frac t{2(1-t)}w^2+a_1w-\frac{vw}{1-t}}
\notag\\
&\times\prod_{j=1}^n\frac{w/\nu_j-1}{\zeta/\nu_j-1}\prod_{j=1}^m(1-\zeta/\nu_{j+n})\frac {e^{ax\zeta}}{\zeta-w}
\notag\\
&+\frac {d_2\sqrt{a}}{2\pi i\sqrt{1-t}}\int_{\Gamma_{c_1}}\,dwe^{\frac t{2(1-t)}w^2+a_1w+\frac{vw}{1-t}+axw}
\prod_{j=1}^m(1-w/\nu_{j+n}).
\end{align}
In the first integral in the right hand side of (\ref{2.66})
we can choose $c_1>b_2$, where $b_2$ is defined in (\ref{1.1''}), and then deform $\gamma'_{1,w}$ to $\gamma_1$. We see that
\begin{equation}
\int_{\Gamma_{c_1}}G_{v,t}(w)b_1^w\,dw=\mathcal{B}_{v,t}(x)+\beta_{v,t}(x),
\notag
\end{equation}
where $\mathcal{B}_{v,t}$ is given by (\ref{1.4}) and $\beta_{v,t}$ by (\ref{1.5}).
It follows that the right hand side of (\ref{1.12}) agrees with the right hand side
of (\ref{2.63}).

For the rewriting of $\hat{L}_{m,n}$ we obtain instead of 
(\ref{2.60}) the kernel
\begin{equation}\label{2.67}
\hat{M}^{(z,w)}=\hat{M}_0+(w-z)\hat{b}_1^w\otimes \hat{b}_2^z,
\end{equation}
where $\hat{M}_0$ is given by (\ref{1.12.6}),
\begin{equation}\label{2.68}
\hat{b}_1^w(x)=\frac{\sqrt{a}}{2\pi i}\int_{\gamma_{2,w}}H_{m,n}(\zeta)^{-1}\frac{e^{-ax\zeta}}{\zeta-w}\,d\zeta
\end{equation}
and
\begin{equation}\label{2.69}
\hat{b}_2^z(y)=\frac{\sqrt{a}}{2\pi i}\int_{\gamma_1}H_{m,n}(\omega)\frac{e^{ay\omega}}{\omega-z}\,d\omega.
\end{equation}
Also, instead of (\ref{2.61}) and (\ref{2.62}) we need
\begin{equation}\label{2.70}
\hat{F}_{u,s}(z)=\frac {d_1}{2\pi i\sqrt{1-s}}e^{-\frac s{2(1-s)}z^2-a_2z+\frac{uz}{1-s}}
\prod_{j=1}^m\left(1-\frac z{\nu_{n+j}}\right)^{-1}
\end{equation}
\begin{equation}\label{2.71}
G_{v,t}(w)=\frac {d_2}{2\pi i\sqrt{1-t}}e^{\frac{t}{2(1-t)}w^2+a_2w-\frac{vw}{1-t}}\prod_{j=1}^m\left(1-\frac w{\nu_{n+j}}
\right).
\end{equation}
From (\ref{2.57}) we obtain
\begin{align}\label{2.72}
d_1d_2 \hat{L}_{m,n}(s,u,t,v;\nu)&=\int_{\gamma_2}\,dz\int_{\Gamma_{c_2}}\,dw \hat{F}_{u,s}(z)
\hat{G}_{v,t}(w)\frac 1{w-z}
\notag\\
&\times\frac {\det(I-\hat{M}_0-(w-z)\hat{b}_1^w\otimes \hat{b}_2^z)_{L^2[1,\infty)}}{\det(I-\hat{M}_0)_{L^2[1,\infty)}}.
\end{align}
Using lemma \ref{lem2.2} again we get
\begin{align}\label{2.73}
&\int_{\gamma_2}\,dz\int_{\Gamma_{c_2}}\,dw \hat{F}_{u,s}(z)
\hat{G}_{v,t}(w)\frac 1{w-z}\det(I-\hat{M}_0-(w-z)\hat{b}_1^w\otimes \hat{b}_2^z)_{L^2[1,\infty)}
\notag\\
&=\left[\int_{\gamma_2}\,dz\int_{\Gamma_{c_2}}\,dw \hat{F}_{u,s}(z)
\hat{G}_{v,t}(w)\frac 1{w-z}-1\right]\det(I-\hat{M}_0)_{L^2[1,\infty)}
\notag\\
&+\det\left(I-\hat{M}_0-\left(\int_{\Gamma_{c_2}}\hat{G}_{v,t}(w)\hat{b}_1^w\,dw\right)
\otimes\left(\int_{\gamma_2}\hat{F}_{u,s}(z)\hat{b}_2^z\,dz\right)
\right)_{L^2[1,\infty)}.
\end{align}

Now, 
\begin{align}\label{2.74}
\int_{\gamma_2}\hat{F}_{u,s}(z)\hat{b}_2^z\,dz&=\frac {d_1\sqrt{a}}{(2\pi i)^2\sqrt{1-s}}\int_{\gamma_2}\,dz
\int_{\gamma_1}\,d\omega e^{-\frac s{2(1-s)}z^2-a_2z+\frac{uz}{1-s}}
\notag\\
&\times\prod_{j=1}^m\frac {1-\omega/\nu_{n+j}}{1-z/\nu_{j+n}}\prod_{j=1}^n\frac 1{\omega/\nu_j-1}
\frac{e^{ay\omega}}{\omega-z}=\hat{\mathcal{C}}_{u,s}(y),
\end{align}
where $\hat{\mathcal{C}}_{u,s}(y)$ is given by (\ref{1.12.5}). Also,
\begin{align}
\int_{\Gamma_{c_2}}\hat{G}_{v,t}(w)\hat{b}_1^w\,dw&=\frac {d_2\sqrt{a}}{(2\pi i)^2\sqrt{1-t}}
\int_{\Gamma_{c_2}}dw\int_{\gamma_{2,w}}\,d\zeta e^{\frac t{2(1-t)}w^2+a_2w-\frac{vw}{1-t}}
\notag\\
&\times\prod_{j=1}^m\frac{1-w/\nu_{n+j}}{1-\zeta/\nu_{n+j}}\prod_{j=1}^n(\zeta/\nu_{j}-1)\frac {e^{-ax\zeta}}{\zeta-w}.
\notag
\end{align}
We can move $\gamma_{2,w}$ to a contour $\gamma'_{2,w}$ in the right half plane that contains
$\nu_{n+1},\dots,\nu_{n+m}$ but does not contain $w$ in its interior. This gives,
\begin{align}\label{2.75}
\int_{\Gamma_{c_2}}\hat{G}_{v,t}(w)\hat{b}_1^w\,dw&=\frac {d_2\sqrt{a}}{(2\pi i)^2\sqrt{1-t}}
\int_{\Gamma_{c_2}}\,dw\int_{\gamma'_{2,w}}\,d\zeta e^{\frac t{2(1-t)}w^2+a_2w-\frac{vw}{1-t}}
\notag\\
&\times\prod_{j=1}^m\frac{1-w/\nu_{n+j}}{1-\zeta/\nu_{n+j}}\prod_{j=1}^n(\zeta/\nu_{j}-1)
\frac {e^{-ax\zeta}}{\zeta-w}
\notag\\
&+\frac {d_2\sqrt{a}}{2\pi i\sqrt{1-t}}\int_{\Gamma_{c_2}}\,dwe^{\frac t{2(1-t)}w^2+a_2w-\frac{vw}{1-t}-axw}
\prod_{j=1}^n(w/\nu_{j}-1).
\end{align}
In the first integral in the right hand side of (\ref{2.75})
we can choose $c_2<b_2$ and then deform $\gamma'_{2,w}$ to $\gamma_2$. We see that
\begin{equation}
\int_{\Gamma_{c_2}}\hat{G}_{v,t}(w)\hat{b}_1^w\,dw=\hat{\mathcal{B}}_{v,t}(x)+\hat{\beta}_{v,t}(x),
\notag
\end{equation}
where $\hat{\mathcal{B}}_{v,t}$ is given by (\ref{1.12.3}) and $\hat{\beta}_{v,t}$ by (\ref{1.12.4}).
We conclude that the right hand side of (\ref{2.72}) agrees with the right hand side
of (\ref{1.12.2}). This completes the proof of theorem \ref{thm1.2}

\section{Asymptotics and the extended tacnode kernel}\label{sect3}

In this section we will prove theorem \ref{thm1.3} by analyzing the
asymptotics of the kernel $\mathcal{L}_n(s,u,t,v)$ which we get from $\mathcal{L}_{m,n}(s,u,t,v;\nu)$
when we make the special choice $m=n$, $a_1=-a/2=-a_2$, $\nu_1=\dots=\nu_n=b_1=-a/2$,
$\nu_{n+1}=\dots=\nu_{n+m}=b_2=a/2$, where $a>0$. By (\ref{1.16}) it is enough to investigate the
asymptotics of $L_{n}(s,u,t,v)$ and $q(s,u,t,v)$ under the scaling given by (\ref{1.17}). Note that,
by (\ref{1.12}),
\begin{align}\label{3.1} 
d_1d_2L_{n}(s,u,t,v)&=\frac{d_1d_2}{(2\pi i)^2\sqrt{(1-s)(1-t)}}\int_{\Gamma_0}\,dw\int_{\gamma_1}\,dz
\frac 1{w-z}
\notag\\
&\times e^{-\frac{s}{2(1-s)}z^2+az/2+\frac{uz}{1-s}+\frac t{2(1-t)}w^2-aw/2-\frac{vw}{1-t}}
\left(\frac{1+2w/a}{1+2z/a}\right)^n
\notag\\
&-1+\frac{\det\left(I-M_0+(\mathcal{B}_{v,t}+\beta_{v,t})\otimes \mathcal{C}_{u,s}\right)_
{L^2[1,\infty)}}{\det(I-M_0)_{L^2[1,\infty)}}.
\end{align}
In this case,
\begin{align}\label{3.2}
\mathcal{B}_{v,t}(x)&=\frac{d_2\sqrt{a}}{(2\pi i)^2\sqrt{1-t}}\int_{\Gamma_0}\,dw
\int_{\gamma_1}\,d\zeta e^{\frac t{2(1-t)}w^2-aw/2-\frac{vw}{1-t}+ax\zeta}\frac 1{\zeta-w}
\notag\\
&\times\left(\frac{1-2\zeta/a}{1+2\zeta/a}\right)^n(1+2w/a)^n,
\end{align}
\begin{equation}\label{3.3}
\beta_{v,t}(x)=\frac{d_2\sqrt{a}}{2\pi i\sqrt{1-t}}\int_{\Gamma_0}\,dw
e^{\frac t{2(1-t)}w^2-aw/2-\frac{vw}{1-t}+axw}(1-2w/a)^n,
\end{equation}
\begin{align}\label{3.4}
\mathcal{C}_{u,s}(y)&=\frac{d_1\sqrt{a}}{(2\pi i)^2\sqrt{1-s}}\int_{\gamma_1}\,dz
\int_{\gamma_2}\,d\omega e^{-\frac s{2(1-s)}z^2+az/2+\frac{uz}{1-s}-ay\omega}\frac 1{\omega-z}
\notag\\
&\times\left(\frac{1+2\omega/a}{1-2\omega/a}\right)^n\frac 1{(1+2z/a)^n},
\end{align}
and
\begin{equation}\label{3.5}
M_0(x,y)=\frac{a}{(2\pi i)^2}\int_{\gamma_1}\,d\zeta\int_{\gamma_2}\,d\omega
\frac{e^{ax\zeta-ay\omega}}{\zeta-\omega}\left(\frac{1+2\omega/a}{1-2\omega/a}\right)^n
\left(\frac{1-2\zeta/a}{1+2\zeta/a}\right)^n.
\end{equation}
Here $\gamma_1$ can be taken to be a circle of radius $<a/2$ around $-a/2$ and $\gamma_2$
a circle of radius $<a/2$ around $a/2$.

The functions $\mathcal{B}_{v,t}$, $\beta_{v,t}$, $\mathcal{C}_{u,s}$ and $M_0$ can all
be expressed in terms of Hermite and Laguerre polynomials. This means that we can use
known asymptotic results for these polynomials in our asymptotic analysis. It would also
be possible to use the above formulas directly in a saddle-point analysis. In (\ref{3.2}) we make the change of variables
$\zeta\to-a\zeta/2$ and $w\to aw/2$. Let $D_1$ be a circle with radius $<1$ around $1$.
Then,
\begin{align}\label{3.6}
\mathcal{B}_{v,t}(x)&=\frac{a^{3/2}d_2}{2(2\pi i)^2\sqrt{1-t}}\int_{\Gamma_0}\,dw
\int_{D_1}\,d\zeta e^{\frac {a^2t}{8(1-t)}w^2-a^2w/4-\frac{avw}{2(1-t)}-a^2x\zeta/2}\frac 1{\zeta+w}
\notag\\
&\times\left(\frac{1+\zeta}{1-\zeta}\right)^n(1+w)^n.
\end{align}
In (\ref{3.3}) we make the change of variables $w\to -aw/2$, which gives
\begin{equation}\label{3.7}
\beta_{v,t}(x)=\frac{a^{3/2}d_2}{(4\pi i\sqrt{1-t})}\int_{\Gamma_0}\,dw
e^{\frac {a^2t}{8(1-t)}w^2+\left(\frac {a^2}4+\frac{av}{2(1-t)}-\frac{a^2x}2\right)w}(1+w)^n.
\end{equation}
Furthermore, in (\ref{3.4}) we make the change of 
variables $\omega\to a\omega/2$, $z\to az/2$. Let $D_{-1}$ be a circle around
$-1$ with radius $<1$. Then,
\begin{align}\label{3.8}
\mathcal{C}_{u,s}(y)&=\frac{a^{3/2}d_1}{2(2\pi i)^2\sqrt{1-s}}\int_{D_{-1}}\,dz
\int_{D_1}\,d\omega e^{-\frac {a^2s}{8(1-s)}z^2+a^2z/4+\frac{auz}{2(1-s)}-a^2y\omega/2}\frac 1{\omega-z}
\notag\\
&\times\left(\frac{1+\omega}{1-\omega}\right)^n\frac 1{(1+z)^n},
\end{align}

Finally, in (\ref{3.5}), we make the change of variables $\omega\to a\omega/2$, $\zeta\to -a\zeta/2$, which gives
\begin{equation}\label{3.9}
M_0(x,y)=\frac{a^2}{2(2\pi i)^2}\int_{D_1}\,d\zeta\int_{D_1}\,d\omega\frac{e^{-a^2(x\zeta+y\omega)/2}}
{\zeta+\omega}\left(\frac{1+\omega}{1-\omega}\right)^n
\left(\frac{1+\zeta}{1-\zeta}\right)^n.
\end{equation}
We will use the following formulas for the Hermite and Laguerre polynomials. Let $h_n(x)$ denote the normalized Hermite
polynomials w.r.t. the weight $e^{-x^2}$ on $\mathbb{R}$. Then
\begin{equation}\label{3.10}
\frac 1{2\pi i}\int_{\Gamma_0}(w+1)^ne^{Aw^2-2Bw}\,dw=\frac{\sqrt{n!}}{\pi^{1/4}\sqrt{2}}
\frac {e^{-B^2/A}}{(\sqrt{2A})^{n+1}}h_n(\frac B{\sqrt{A}}+\sqrt{A}),
\end{equation}
and
\begin{equation}\label{3.11}
\frac 1{2\pi i}\int_{D_{-1}}\frac{e^{-Az^2+2Bz}}{(z+1)^n}\,dz
=\frac{\pi^{1/4}(\sqrt{2A})^{n-1}}{\sqrt{(n-1)!}}e^{-A-2B}h_{n-1}(\frac B{\sqrt{A}}+\sqrt{A}).
\end{equation}
These formulas follow from the classical integral representations of the Hermite polynomials. Let
$\ell_n^1(x)$ denote the normalized Laguerre polynomial with weight $xe^{-x}$ on $[0,\infty)$ and with 
positive leading coefficient. Then,
\begin{equation}\label{3.12}
\frac 1{2\pi i}\int_{D_1}e^{-x\zeta/2}\left(\frac{1+\zeta}{1-\zeta}\right)^n\,d\zeta=-2\sqrt{n}e^{-x/2}\ell_{n-1}^1(x),
\end{equation}
as can be derived from the contour integral formula for Laguerre polynomials. In (\ref{3.6}) we can write
\begin{equation}
\frac 1{\zeta+w}=\int_0^\infty e^{-\lambda(\zeta+w)}\,d\lambda,
\notag
\end{equation}
since $\re (\zeta+w)>0$. Hence, by (\ref{3.10}) and (\ref{3.12}),
\begin{align}\label{3.13}
&\mathcal{B}_{v,t}(x)=\frac{a^{3/2}d_2}{2\sqrt{1-t}}\int_0^\infty\left(
\frac 1{2\pi i}\int_{\Gamma_0}(w+1)^ne^{\frac {a^2t}{8(1-t)}w^2-2(a^2/8+\frac{av}{4(1-t)}+\lambda/2)w}\,dw\right)
\notag\\
&\times\left(\frac 1{2\pi i}\int_{D_1}e^{-(a^2x+2\lambda)\zeta/2}\left(\frac{1+\zeta}{1-\zeta}\right)^n\,d\zeta\right)
\,d\lambda 
\notag\\
&=-\frac{a^{3/2}d_2\sqrt{n!}\sqrt{n}}{\sqrt{2}\pi^{1/4}\sqrt{1-t}}\left(\frac{a^2t}{4(1-t)}\right)^{-(n+1)/2}
\int_0^\infty e^{-\frac{8(1-t)}{a^2t}\left(a^2/8+\frac{av}{4(1-t)}+\lambda/2\right)^2}
\notag\\
&\times h_n\left(\sqrt{\frac{8(1-t)}{a^2t}}\left(\frac{a^2}8+\frac{av}{4(1-t)}+\frac{\lambda}2\right)+\sqrt{\frac{a^2t}{8(1-t)}}\right)
e^{-a^2x/2-\lambda} \ell_{n-1}^1(a^2x+2\lambda)\,d\lambda.
\end{align}
In (\ref{3.7}) we can use (\ref{3.10}) directly and obtain
\begin{align}\label{3.14}
\beta_{v,t}(x)&=
\frac{a^{3/2}d_2\sqrt{n!}}{2\pi^{1/4}\sqrt{2(1-t)}}\left(\frac{a^2t}{4(1-t)}\right)^{-(n+1)/2}
e^{-\frac{8(1-t)}{a^2t}\left(\frac{a^2x}{4}-\frac{a^2}8-\frac{av}{4(1-t)}\right)^2}
\notag\\
&\times h_n\left(\sqrt{\frac{8(1-t)}{a^2t}}\left(\frac{a^2x}{4}-\frac{a^2}8-\frac{av}{4(1-t)}\right)
+\sqrt{\frac{a^2t}{8(1-t)}}\right).
\end{align}
Next, in (\ref{3.8}) we use (\ref{3.11}), (\ref{3.12}) and
\begin{equation}
\frac 1{\omega-z}=\int_0^\infty e^{-\lambda(\omega-z)}\,d\lambda,
\notag
\end{equation}
to get
\begin{align}\label{3.15}
&\mathcal{C}_{u,s}(y)=
-\frac{a^{3/2}d_1\sqrt{n}\pi^{1/4}}{\sqrt{(n-1)!}\sqrt{1-s}}\left(\frac{a^2s}{4(1-s)}\right)^{-(n-1)/2}
\int_0^\infty e^{-\frac{a^2s}{8(1-s)}-\frac{a^2}4-\frac{au}{2(1-s)}-\lambda}
\notag\\
&\times h_{n-1}\left(\sqrt{\frac{8(1-s)}{a^2s}}\left(\frac{a^2}8+\frac{au}{4(1-s)}
+\frac{\lambda}2\right)+\sqrt{\frac{a^2s}{8(1-s)}}\right)
e^{-a^2y/2-\lambda} \ell_{n-1}^1(a^2y+2\lambda)\,d\lambda.
\end{align}
Finally, in (\ref{3.9}) we use $(\zeta+\omega)^{-1}=\int_0^\infty \exp(-\lambda(\zeta+\omega))\,d\lambda$
to get
\begin{equation}\label{3.16}
M_0(x,y)=2a^2n\int_0^\infty e^{-(a^2x/2+\lambda)}\ell_{n-1}^1(a^2x+2\lambda)
e^{-(a^2y/2+\lambda)}\ell_{n-1}^1(a^2y+2\lambda)\,d\lambda
\end{equation}

We would now like to insert the scalings (\ref{1.17}) into (\ref{3.1}) and take the limit as $n\to\infty$. For
this we use the following asymptotic results for the Hermite and Laguerre polynomials.

\begin{lemma}\label{lem3.1}
Uniformly for $\xi$ in a compact set we have the limit
\begin{equation}\label{3.17}
n^{1/12}h_n\left(\sqrt{2n}\left(1+\frac {\xi}{n^{2/3}}\right)\right)
e^{-n(1+\xi/n^{2/3})^2}\to 2^{1/4}\Ai(2\xi)
\end{equation}
as $n\to\infty$. Also, there are constants $c,C>0$ so that
\begin{equation}\label{3.18}
\left|n^{1/12}h_n\left(\sqrt{2n}\left(1+\frac {\xi}{n^{2/3}}\right)\right)
e^{-n(1+\xi/n^{2/3})^2}\right|\le Ce^{-c\xi^{3/2}}
\end{equation}
for all $\xi\ge 0$, and
\begin{equation}\label{3.19}
\left|n^{1/12}h_n\left(\sqrt{2n}\left(1+\frac {\xi}{n^{2/3}}\right)\right)
e^{-n(1+\xi/n^{2/3})^2}\right|\le Ce^{-\xi^2/3n^{1/3}}
\end{equation}
for all $\xi\ge n^{2/3}$.
Furthermore, uniformly for $\xi$ in a compact set,
\begin{equation}\label{3.20}
n^{5/6}\ell_n^1\left(4n\left(1+\frac {\xi}{n^{2/3}}\right)\right)e^{-2n(1+\xi/n^{2/3})}\to2^{-4/3}\Ai(2^{2/3}\xi)
\end{equation}
as $n\to\infty$. There are constants $c,C>0$, so that
\begin{equation}\label{3.21}
\left|n^{5/6}\ell_n^1\left(4n\left(1+\frac {\xi}{n^{2/3}}\right)\right)e^{-2n(1+\xi/n^{2/3})}\right|
\le Ce^{-c\xi},
\end{equation}
for all $\xi\ge 0$.
\end{lemma}

\begin{proof}
Fix $\delta>0$ small and let
\begin{equation}
F(x)=\left|\int_x^1\sqrt{|1-y^2|^2}\,dy\right|.
\notag
\end{equation}
From \cite{De99}, theorem 2.2, we have the following asymptotic formulas for the normalized Hermite polynomials
\begin{align}\label{3.21.1}
h_n(\sqrt{2n} x)e^{-nx^2}&=(2n)^{-1/4}\left\{\left(\frac{1+x}{1-x}\right)^{1/4}[3nF(x)]^{1/6}\Ai(-[3nF(x)]^{2/3})\right.
\notag\\
&-\left(\frac{1-x}{1+x}\right)^{1/4}\left.[3nF(x)]^{-1/6}\Ai'(-[3nF(x)]^{2/3})\right\}(1+\mathcal{O}(n^{-1})),
\end{align}
for $1-\delta\le x\le 1$,
\begin{align}\label{3.21.2}
h_n(\sqrt{2n} x)e^{-nx^2}&=(2n)^{-1/4}\left\{\left(\frac{x+1}{x-1}\right)^{1/4}[3nF(x)]^{1/6}\Ai([3nF(x)]^{2/3})\right.
\notag\\
&-\left(\frac{x-1}{x+1}\right)^{1/4}\left.[3nF(x)]^{-1/6}\Ai'([3nF(x)]^{2/3})\right\}(1+\mathcal{O}(n^{-1})),
\end{align}
for $1\le x\le 1+\delta$ and
\begin{equation}\label{3.21.3}
\left|h_n(\sqrt{2n} x)e^{-nx^2}\right|\le Cn^{-1/4} e^{-nF(x)}
\end{equation}
for $x\ge 1+\delta$. These asymptotic formulas and estimates together with asymptotics for the Airy function and 
its derivative can be used to prove (\ref{3.17}) to (\ref{3.19}). Let us give the proof of (\ref{3.19}).
If $\xi\ge n^{2/3}$, then $x=1+\xi/n^{2/3}\ge 2$ so we can use the estimate (\ref{3.21.3}). We obtain
\begin{align}
&\left|n^{1/12}h_n\left(\sqrt{2n}\left(1+\frac {\xi}{n^{2/3}}\right)\right)
e^{-n(1+\xi/n^{2/3})^2}\right|
\notag\\
&\le \frac C{n^{1/6}}\exp\left(-n\int_1^{1+\xi/n^{2/3}}\sqrt{t^2-1}\,dt\right)
\notag
\end{align}
for all $\xi\ge n^{2/3}$. Now,
\begin{align}
&\int_1^{1+\xi/n^{2/3}}\sqrt{t^2-1}\,dt\ge\int_{\sqrt{2}}^{1+\xi/n^{2/3}} t\sqrt{1-1/t^2}\,dt\ge\frac 1{\sqrt{2}}
\int_{\sqrt{2}}^{1+\xi/n^{2/3}} t\,dt
\notag\\
&=\frac 1{2\sqrt{2}}\left(\left(1+\frac{\xi}{n^{2/3}}\right)^2-2\right)=\frac 1{2\sqrt{2}}
\left(\frac{2\xi}{n^{2/3}}-1+\frac{\xi^2}{n^{4/3}}\right)\ge \frac 1{2\sqrt{2}}\frac{\xi^2}{n^{4/3}}
\notag
\end{align}
if $\xi\ge n^{2/3}$. This proves (\ref{3.19}). 

In \cite{Va05} formulas similar to (\ref{3.21.1}) to (\ref{3.21.3}) are given for generalized Laguerre
polynomials and specializing to the Laguerre polynomilas we are considering we can use the results of \cite{Va05}
to prove (\ref{3.20}) and (\ref{3.21}). (In the notation of \cite{Va05} we are considering the case
when $\alpha=1$, $Q(x)=x$, which gives $h_n(x)=4$, $d\mu_n(x)=2\pi^{-1}\sqrt{(1-x)/x}$, $\beta_n=4n$ and
$\psi_n(z)=2(\pi i)^{-1}(z-1)^{1/2}/z^{1/2}$. We can then use Theorem 2.4 in \cite{Va05}.)
\end{proof}

These asymptotic results and estimates can be used to prove the next lemma, which contains the essential asymptotic results
that we need.

\begin{lemma}\label{lem3.2}
Consider the scaling (\ref{1.17}) and choose $d_1,d_2$ as in (\ref{1.18}). We have the following pointwise limits,
\begin{equation}\label{3.22}
\lim_{n\to\infty} n^{-1/3}\mathcal{B}_{v,t}(1+xn^{-2/3})=-\sqrt{2}\int_0^\infty e^{2^{1/3}\tau_2\lambda}
\Ai(\xi_2+\tau_2^2+\sigma+2^{1/3}\lambda)\Ai(2^{2/3}x+\tilde{\sigma}+\lambda)\,d\lambda,
\end{equation}
\begin{equation}\label{3.23}
\lim_{n\to\infty} n^{-1/3}\beta_{v,t}(1+xn^{-2/3})=\sqrt{2}e^{2\tau_2(x-\xi_2)}
\Ai(-\xi_2+\sigma+2x+\tau_2^2),
\end{equation}
\begin{equation}\label{3.24}
\lim_{n\to\infty} n^{-1/3}\mathcal{C}_{u,s}(1+yn^{-2/3})=-\sqrt{2}\int_0^\infty e^{-2^{1/3}\tau_1\lambda}
\Ai(\xi_1+\tau_1^2+\sigma+2^{1/3}\lambda)\Ai(2^{2/3}y+\tilde{\sigma}+\lambda)\,d\lambda
\end{equation}
and
\begin{equation}\label{3.25}
\lim_{n\to\infty} n^{-2/3}M_0(1+xn^{-2/3},1+yn^{-2/3})= 2^{2/3}K_{\text{Ai}}(2^{2/3}x+\tilde{\sigma},
2^{2/3}y+\tilde{\sigma}).
\end{equation}
We also have the following estimates. There are constants $c,C>0$ so that, for all $x,y\ge 0$,
\begin{equation}\label{3.26}
\left|n^{-1/3}\mathcal{B}_{v,t}(1+xn^{-2/3})\right|\le Ce^{-cx},
\end{equation}
\begin{equation}\label{3.27}
\left|n^{-1/3}\beta_{v,t}(1+xn^{-2/3})\right|\le Ce^{-cx},
\end{equation}
\begin{equation}\label{3.28}
\left|n^{-1/3}\mathcal{C}_{u,s}(1+yn^{-2/3})\right|\le Ce^{-cy},
\end{equation}
and
\begin{equation}\label{3.29}
\left| n^{-2/3}M_0(1+xn^{-2/3},1+yn^{-2/3})\right|\le Ce^{-c(x+y)}.
\end{equation}
\end{lemma}

\begin{proof} Set
\begin{align}\label{3.30}
X_n(\lambda)&=\sqrt{\frac{8(1-t)}{a^2t}}\left(\frac{a^2}8+\frac{av}{4(1-t)}+\frac {\lambda}2\right)+\sqrt{\frac{a^2t}{8(1-t)}}
\notag\\
&=\sqrt{2n}\left(1+\frac a{2\sqrt{n}}\frac{\sqrt{t}+\sqrt{1-t}}{\sqrt{4(1-t)}}-1+\frac v{2\sqrt{n}}
\sqrt{\frac{1-t}t}+\frac{\lambda}{a\sqrt{n}}\sqrt{\frac{1-t}t}\right)
\end{align}
and
\begin{align}\label{3.31}
Y_n(\lambda)&=-\frac{8(1-t)}{a^2t}\left(\frac{a^2}8+\frac{av}{4(1-t)}+\frac {\lambda}2\right)^2+\frac 12X_n(\lambda)^2
\notag\\
&=\frac {a^2}8+\frac{a^2(2t-1)}{16t(1-t)}+\left(\frac{2t-1}{2t}-\frac v{at}\right)\lambda-\frac{(1-t)\lambda^2}{a^2t}
-\frac{v^2}{4t(1-t)}+\frac{av(2t-1)}{4t(1-t)}.
\end{align}
Also set
\begin{equation}\label{3.32}
Z_n=\frac{\sqrt{n!}n^{1/3}}{\pi^{1/4}\sqrt{2}}\left(\frac{a^2t}{4(1-t)}\right)^{-(n+1)/2}.
\end{equation}
Furthermore, set
\begin{equation}\label{3.33}
W_n(\lambda)=a^2(1+xn^{-2/3})+2\lambda=4n\left[1+\frac{a^2}{4n}-1+\frac{a^2}{4n}xn^{-2/3}+\frac{\lambda}{2n}\right]
\end{equation}
Then, by (\ref{3.13}),
\begin{align}\label{3.34}
&n^{-1/3}\mathcal{B}_{v,t}(1+xn^{-2/3})
\notag\\
&=-\frac{a^{3/2}d_2}{n\sqrt{1-t}}\int_0^\infty Z_ne^{Y_n(\lambda)-X_n(\lambda)^2/2}
h_n(X_n(\lambda))n^{5/6}e^{-W_n(\lambda)/2}\ell_{n-1}^1(W_n(\lambda))\,d\lambda.
\end{align}
Now, using (\ref{3.17}), (\ref{3.30}), (\ref{3.31}), (\ref{3.32}) and Stirling's formula a somewhat 
lengthy but straightforward computation shows that
\begin{equation}\label{3.35}
Z_ne^{Y_n(\lambda n^{1/3})-X_n(\lambda n^{1/3})^2/2}h_n(X_n(\lambda n^{1/3}))\to
e^{\tau_2(\xi_2+\sigma+\lambda)+2\tau_2^3/3}\Ai(\xi_2+\tau_2^2+\sigma+\lambda)
\end{equation}
pointwise as $n\to\infty$. Also, from (\ref{3.30}), (\ref{3.31}), (\ref{3.32}) and  (\ref{3.18}) it follows 
that there are constants $c,C>0$ so that
\begin{equation}\label{3.36}
\left|Z_ne^{Y_n(\lambda n^{1/3})-X_n(\lambda n^{1/3})^2/2}h_n(X_n(\lambda n^{1/3}))\right|
\le Ce^{-c\lambda^{3/2}}
\end{equation}
for all $\lambda\ge 0$. From (\ref{3.20}) and (\ref{3.33}) it follows that
\begin{equation}\label{3.37}
n^{5/6}e^{-W_n(\lambda n^{1/3})/2}\ell_{n-1}^1(W_n(\lambda n^{1/3}))\to 2^{-4/3}\Ai(2^{2/3}(x+\sigma+\lambda/2))
\end{equation}
pointwise as $n\to\infty$. Furthermore it follows from (\ref{3.21}) and (\ref{3.33}) that
there are constants $c,C>0$ so that
\begin{equation}\label{3.38}
\left|n^{5/6}e^{-W_n(\lambda n^{1/3})/2}\ell_{n-1}^1(W_n(\lambda n^{1/3}))\right|
\le Ce^{-c(x+\lambda)}.
\end{equation}
We now make the change of variables $\lambda\to\lambda n^{1/3}$ in (\ref{3.34}) and use (\ref{1.17}), (\ref{1.19})
to get
\begin{align}\label{3.39}
&n^{-1/3}\mathcal{B}_{v,t}(1+xn^{-2/3})=-\frac{(2+\sigma n^{-2/3})^{3/2}e^{-\tau_2(\xi_2+\sigma)-2\tau_2^3/3}}
{\sqrt{2(1-t)}}
\notag\\
&\times \int_0^\infty Z_ne^{Y_n(\lambda n^{1/3})-X_n(\lambda n^{1/3})^2/2}
h_n(X_n(\lambda n^{1/3}))n^{5/6}e^{-W_n(\lambda n^{1/3})/2}\ell_{n-1}^1(W_n(\lambda n^{1/3}))\,d\lambda.
\end{align}
Using (\ref{3.35}) to (\ref{3.38}) and the dominated convergence theorem we see that
\begin{align}
&n^{-1/3}\mathcal{B}_{v,t}(1+xn^{-2/3})\to -2^{3/2-4/3}\int_0^\infty e^{\tau_2\lambda}
\Ai(\xi_2+\tau_2^2+\sigma+\lambda)\Ai(2^{2/3}(x+\sigma+\lambda/2))\,d\lambda
\notag\\
&=-\sqrt{2}\int_0^\infty e^{2^{1/3}\tau_2\lambda}\Ai(\xi_2+\tau_2^2+\sigma+2^{1/3}\lambda)\Ai(2^{2/3}x+\tilde{\sigma}+\lambda)\,d\lambda,
\notag
\end{align}
which proves (\ref{3.22}). If we use the estimates (\ref{3.36}) and (\ref{3.38}) in (\ref{3.39}) we
obtain the estimate (\ref{3.26}).

Consider now $\beta_{v,t}(x)$. Note that
\begin{equation}
\frac{a^2(1+xn^{-2/3})}4-\frac{a^2}8-\frac{av}{4(1-t)}=\frac{a^2}8-\frac{av}{4(1-t)}+\frac{a^2}4 xn^{-2/3}+o(n^{-2/3}).
\notag
\end{equation}
If we set $\tilde{\lambda}=a^2x/2n+o(n^{-1})$, then we see from (\ref{3.14}) that
\begin{equation}\label{3.40}
n^{-1/3}\beta_{-v,t}(1+xn^{-2/3})=\frac{a^{3/2}d_2n^{-2/3}}{2\sqrt{1-t}} Z_ne^{Y_n(\tilde{\lambda}n^{1/3})-
X_n(\tilde{\lambda}n^{1/3})^2/2}h_n(X_n(\tilde{\lambda}n^{1/3})).
\end{equation}
We can now use (\ref{3.35}) and the definition of $d_2$ to get the following limit. Note that changing $v$
to $-v$ corresponds to changing $\xi_2$ to $-\xi_2$ in (\ref{3.35}). This gives
\begin{equation}
\lim_{n\to\infty} n^{-1/3}\beta_{v,t}(1+xn^{-2/3})=
\sqrt{2} e^{-2\tau_2\xi_2+2\tau_2x}\Ai(2x-\xi_2+\sigma+\tau_2^2),
\notag
\end{equation}
which proves (\ref{3.23}). The estimate (\ref{3.27}) follows from
(\ref{3.36}) and (\ref{3.40}).

Next, we turn to $\mathcal{C}_{u,s}$. We proceed similarly to the analysis of $\mathcal{B}_{v,t}$. Set
\begin{align}\label{3.41}
&\hat{X}_n(\lambda)=\sqrt{\frac{8(1-s)}{a^2s}}\left(\frac{a^2}8+\frac{au}{4(1-s)}+\frac{\lambda}2\right)+
\sqrt{\frac{a^2s}{8(1-s)}}
\notag\\
&=\sqrt{2(n-1)}\left[1+\frac{a}{2\sqrt{n-1}}\frac{\sqrt{s}+\sqrt{1-s}}{\sqrt{4(1-s)}}-1+
\frac{u}{2\sqrt{n-1}}\sqrt{\frac{1-s}s}+\frac{\lambda}{a\sqrt{n-1}}\sqrt{\frac{1-s}s}\right],
\end{align}
\begin{align}\label{3.42}
&\hat{Y}_n(\lambda)=-\frac{a^2s}{8(1-s)}-\frac{a^2}4-\frac{au}{2(1-s)}
\notag\\
&+\frac 12
\left[\sqrt{\frac{8(1-s)}{a^2s}}\left(\frac{a^2}8+\frac{au}{4(1-s)}+\frac{\lambda}2\right)+
\sqrt{\frac{a^2s}{8(1-s)}}\right]^2
\notag\\
&=-\frac{a^2}8-\frac{a^2(2s-1)}{16s(1-s)}-\left(\frac{2s-1}{2s}-\frac u{as}\right)\lambda
-\frac{au(2s-1)}{4s(1-s)}+\frac{u^2}{4s(1-s)}+\frac{(1-s)\lambda^2}{a^2s}
\end{align}
and
\begin{equation}\label{3.43}
\hat{Z}_n=\left(\frac{a^2s}{4(1-s)}\right)^{(n-1)/2}\frac{\pi^{1/4}n^{1/3}}{\sqrt{(n-1)!}}.
\end{equation}
Now, similarly to (\ref{3.35}) and using (\ref{3.17}), we obtain
\begin{equation}\label{3.44}
\hat{Z}_ne^{\hat{Y}_n(\lambda n^{1/3})-\hat{X}_n(\lambda n^{1/3})^2/2}h_{n-1}(\hat{X}_n(\lambda n^{1/3}))\to
e^{-\tau_1(\xi_1+\sigma+\lambda)-2\tau_1^3/3}\Ai(\xi_1+\tau_1^2+\sigma+\lambda),
\end{equation}
pointwise as $n\to\infty$. We can also get the estimate analogous to (\ref{3.36})
\begin{equation}\label{3.45}
\left|\hat{Z}_ne^{\hat{Y}_n(\lambda n^{1/3})-\hat{X}_n(\lambda n^{1/3})^2/2}h_{n-1}(\hat{X}_n(\lambda n^{1/3}))\right|
\le Ce^{-c\lambda^{3/2}}
\end{equation}
for some constants $c,C>0$. The added difficulty in proving this compared to (\ref{3.36}) is that we now have a factor
$\exp((1-s)\lambda^2n^{2/3}/a^2s)$ instead of $\exp(-(1-t)\lambda^2n^{2/3}/a^2t)$. For large $\lambda$ we cannot ignore
this factor. We have that $(1-s)\lambda^2n^{2/3}/a^2s\sim\lambda^2/4n^{1/3}$ for large $n$. Hence we can use (\ref{3.19})
instead of (\ref{3.18}) when $\lambda\ge c_0n^{2/3}$ with an appropriate $c_0$.

Now, by (\ref{3.15}) and (\ref{3.41}) to (\ref{3.43}),
\begin{align}
&n^{-1/3}\mathcal{C}_{u,s}(1+yn^{-2/3})
\notag\\
&=-\frac{a^{3/2}d_1n^{-2/3}}{\sqrt{1-s}}\int_0^\infty 
\hat{Z}_ne^{\hat{Y}_n(\lambda n^{1/3})-\hat{X}_n(\lambda n^{1/3})^2/2}h_{n-1}(\hat{X}_n(\lambda n^{1/3}))
\notag\\
&\times n^{5/6}e^{-W_n(\lambda n^{1/3})/2}\ell_{n-1}(W_n(\lambda n^{1/3}))\,d\lambda.
\end{align}

It follows from (\ref{3.37}), (\ref{3.38}),(\ref{3.44}), (\ref{3.45}) and the dominated convergence theorem that
\begin{equation}
\lim_{n\to\infty}n^{-1/3}\mathcal{C}_{u,s}(1+yn^{-2/3})=
-\sqrt{2}\int_0^\infty e^{-2^{1/3}\tau_1\lambda}\Ai(\xi_1+\tau_1^2+\sigma+2^{1/3}\lambda)\Ai(2^{2/3}y+\tilde{\sigma}
+\lambda)\,d\lambda,
\notag
\end{equation}
pointwise.
This proves (\ref{3.24}). The estimate (\ref{3.28})
follows from (\ref{3.38}) and (\ref{3.45}).

It remains to consider $M_0$. Here, we note that by (\ref{3.16}) and (\ref{3.33}),
\begin{align}
&n^{-2/3}M_0(1+xn^{-2/3},1+yn^{-2/3})=\frac{2a^2}n\int_0^\infty
n^{5/6}e^{-W_n(\lambda n^{1/3})/2}\ell_{n-1}(W_n(\lambda n^{1/3}))
\notag\\
&\times n^{5/6}e^{-\hat{W}_n(\lambda n^{1/3})/2}\ell_{n-1}(\hat{W}_n(\lambda n^{1/3}))\,d\lambda,
\notag
\end{align}
where $\hat{W}_n$ is the same as $W_n$ but with $x$ replaced by $y$. The limit (\ref{3.25})
and the estimate (\ref{3.29}) now follow from (\ref{3.37}), (\ref{3.38}) and the dominated
convergence theorem.
\end{proof}

We are now ready for the proof of theorem \ref{thm1.3}.

\begin{proof} (\it of theorem \ref{thm1.3}\rm). Note that by (\ref{1.16}) it is enough to show that
\begin{equation}\label{3.46}
\lim_{n\to\infty} d_1d_2 L_n(s,u,t,v)=L_{\text{tac}}(\tau_1,\xi_1,\tau_2,\xi_2)
\end{equation}
and
\begin{equation}\label{3.47}
\lim_{n\to\infty} d_1d_2 q(s,u,t,v)=p(\tau_1,\xi_1,\tau_2,\xi_2)
\end{equation}
under the scaling limit (\ref{1.17}), (\ref{1.18}). If we accept these limits we can complete the proof.
Set
\begin{align}
\hat{d}_1&=\frac{n^{-1/12}}{\sqrt{2}}e^{\tau_1(\sigma-\xi_1)+\frac 23\tau_1^3}
\notag\\
\hat{d}_2&=\frac{n^{-1/12}}{\sqrt{2}}e^{-\tau_2(\sigma-\xi_2)-\frac 23\tau_2^3}.
\notag
\end{align}
Then, by (\ref{3.46}), 
\begin{equation}
\lim_{n\to\infty}\hat{d}_1\hat{d}_2 L_n(s,-u,t,-v)=L_{\text{tac}}(\tau_1,-\xi_1,\tau_2,-\xi_2).
\notag
\end{equation}
By (\ref{1.14}) and (\ref{1.16}),
\begin{align}
&d_1d_2\mathcal{L}_{n}(s,u,t,v)
\notag\\
&=d_1d_2L_n(s,u,t,v)+e^{2\tau_1\xi_1-2\tau_2\xi_2}\hat{d}_1\hat{d}_2L_n(s,-u,t,-v)
+d_1d_2q(s,u,t,t,v)
\notag
\end{align}
and (\ref{1.27}) follows.

It is straightforward to show (\ref{3.47}) from the
definitions of $d_1$, $d_2$, $q$ and $p$, so we omit the proof. To prove (\ref{3.46}) we will use (\ref{3.1}).
Consider the first integral in (\ref{3.1}). This is actually an extended Hermite kernel so we could use known results
but since we have all the ingredients we give a proof.
In this integral we make the change of variables $z\to az/2$, $w\to aw/2$, which gives
\begin{align}\label{3.48}
&\frac{d_1d_2a}{2(2\pi i)^2\sqrt{(1-s)(1-t)}}\int_{\Gamma_0}\,dw\int_{D_{-1}}\,dz\frac 1{w-z}
e^{-\frac{a^2s}{8(1-s)}z^2+\frac{a^2}4 z+\frac{auz}{2(1-s)}}
\notag\\
&\times e^{\frac{a^2t}{8(1-t)}w^2-\frac{a^2}4 w-\frac{auw}{2(1-t)}} \left(\frac{1+w}{1+z}\right)^n
\notag\\
&=\frac{d_1d_2an^{-1/3}}{2\sqrt{(1-s)(1-t)}}\int_0^\infty Z_ne^{Y_n(\lambda n^{1/3})-X_n(\lambda n^{1/3})^2/2}
h_n(X_n(\lambda n^{1/3}))
\notag\\
&\times\hat{Z}_ne^{\hat{Y}_n(\lambda n^{1/3})-\hat{X}_n(\lambda n^{1/3})^2/2}h_{n-1}(\hat{X}_n(\lambda n^{1/3}))\,d\lambda.
\end{align}
Here we have have used (\ref{3.10}), (\ref{3.11}) and the notation in the proof of lemma \ref{lem3.2}. We can now use
(\ref{1.17}), (\ref{1.18}), (\ref{3.35}), (\ref{3.36}), (\ref{3.44}), (\ref{3.45}) and the dominated convergence theorem to see
that the last expression in (\ref{3.48}) converges to
\begin{equation}\label{3.49}
\int_0^\infty e^{\lambda(\tau_2-\tau_1)}\Ai(\xi_1+\tau_1^2+\sigma+\lambda)\Ai(\xi_2+\tau_2^2+\sigma+\lambda)\,d\lambda
=\tilde{A}(\tau_1,\xi_1+\tau_1^2+\sigma,\tau_2,\xi_2+\tau_2^2+\sigma).
\end{equation}
It remains to show that
\begin{align}\label{3.50}
&\frac{\det(I-M_0+(\mathcal{B}_{v,t}+\beta_{v,t})\otimes\mathcal{C}_{u,s})_{L^2[1,\infty)}}
{\det(I-M_0)_{L^2[1,\infty)}}
\notag\\
&\to\frac{1}{F_2(\tilde{\sigma})}\det(I-K_{\text{Ai}}+(B_{\xi_2+\sigma,\tau_2}-\beta_{\xi_2+\sigma,\tau_2})\otimes
B_{\xi_1+\sigma,-\tau_1})_{L^2[\tilde{\sigma},\infty)}
\end{align}
as $n\to\infty$. Write 
\begin{equation}
\mathcal{D}_{v,t}(x,y)=M_0(x,y)-(\mathcal{B}_{v,t}(x)+\beta_{v,t}(x))\mathcal{C}_{u,s}(y)
\notag
\end{equation}
and consider the Fredholm expansion of the numerator in the left hand side of (\ref{3.50}),
\begin{equation}
\sum_{m=0}^\infty\frac{(-1)^m}{m!}\int_{[1,\infty)^m}\det(\mathcal{D}_{v,t}(\rho_i,\rho_j))_{m\times m}\,d^m\rho.
\notag
\end{equation}
Here we make the change of variables $\rho_j=1+x_jn^{-2/3}$ to get
\begin{equation}
\sum_{m=0}^\infty\frac{(-1)^m}{m!}\int_{[0,\infty)^m}\det(n^{-2/3}\mathcal{D}_{v,t}(1+x_in^{-2/3},1+x_jn^{-2/3}))_{m\times m}\,d^m x.
\notag
\end{equation}
Using lemma \ref{lem3.2}, Hadamard's inequality and the dominated convergence theorem we see that, as $n\to\infty$,
this converges to
\begin{equation}\label{3.51}
\sum_{m=0}^\infty\frac{(-1)^m}{m!}\int_{[0,\infty)^m}\det(D(x_i,x_j))_{m\times m}\,d^m x,
\end{equation}
where
\begin{align}\label{3.52}
D(x,y)&=2^{2/3} K_{\text{Ai}}(2^{2/3}x+\tilde{\sigma},2^{2/3}y+\tilde{\sigma})
\notag\\
&+\left(\sqrt{2}\int_0^\infty e^{2^{1/3}\tau_2\lambda}\Ai(\xi_2+\tau_2^2+\sigma+2^{1/3}\lambda)
\Ai(2^{2/3}x+\tilde{\sigma}+\lambda)\,d\lambda
\right.
\notag\\
&-\left.\sqrt{2}e^{-2\tau_2\xi_2+2\tau_2x}\Ai(-\xi_2+\sigma+\tau_2^2+2x)\right)
\notag\\
&\times\int_0^\infty e^{-2^{1/3}\tau_1\lambda}\Ai(\xi_1+\tau_1^2+\sigma+2^{1/3}\lambda)\Ai(2^{2/3}y+\tilde{\sigma}+\lambda)\,d\lambda.
\end{align}
If we make the change of variables $x_i=2^{-2/3}(y_i-\tilde{\sigma})$ in (\ref{3.51}) we see that (\ref{3.51}) is the Fredholm 
expansion of
\begin{equation}
\det(I-K_{\text{Ai}}+(B_{\xi_2+\sigma,\tau_2}-\beta_{\xi_2+\sigma,\tau_2})\otimes
B_{\xi_1+\sigma,-\tau_1})_{L^2[\tilde{\sigma},\infty)}.
\notag
\end{equation}
A similar argument shows that
\begin{equation}
\det(I-M_0)_{L^2[1,\infty)}\to F_2(\tilde{\sigma})
\notag
\end{equation}
as $n\to\infty$. This completes the proof of theorem \ref{thm1.3}.
\end{proof}

We turn now to the proof of the alternative form of $L_{\text{tac}}$, i.e. proposition \ref{prop1.4}.

\begin{proof} (\it of proposition \ref{prop1.4})\rm. All operators are operators on the space $L^2[\tilde{\sigma},\infty)$.
We have that
\begin{align}\label{3.53}
&\frac 1{F_2(\tilde{\sigma})}\det(I-K_{\text{Ai}}+(B_{\xi_2+\sigma,\tau_2}-\beta_{\xi_2+\sigma,\tau_2})\otimes
B_{\xi_1+\sigma,-\tau_1})_{L^2[\tilde{\sigma},\infty)}-1
\notag\\
&=\int_{\tilde{\sigma}}^\infty (B_{\xi_2+\sigma,\tau_2}(x)-\beta_{\xi_2+\sigma,\tau_2}(x))B_{\xi_1+\sigma,-\tau_1}(x)\,dx
\notag\\
&+\int_{\tilde{\sigma}}^\infty\int_{\tilde{\sigma}}^\infty R(x,y)
(B_{\xi_2+\sigma,\tau_2}(y)-\beta_{\xi_2+\sigma,\tau_2}(y))B_{\xi_1+\sigma,-\tau_1}(x)\,dxdy.
\end{align}
Let $T(x,y)=\Ai(x+y-\tilde{\sigma})$ as a kernel on $L^2[\tilde{\sigma},\infty)$. Then $K_{\text{Ai}}=T^2$ and
\begin{equation}
R=\sum_{r=1}^\infty T^{2r}.
\notag
\end{equation}
Set
\begin{equation}
S_{\xi,\tau}(x)=2^{1/6}e^{-2\sigma\tau+2^{1/3}\tau x}\Ai(\xi+\tau^2-\sigma+2^{1/3}x).
\notag
\end{equation}
Then,
\begin{equation}
B_{\xi+\sigma,\tau}(x)=\int_{\tilde{\sigma}}^\infty T(x,y)S_{\xi,\tau}(y)\,dy=\int_{\tilde{\sigma}}^\infty S_{\xi,\tau}(y)T(y,x)\,dy.
\notag
\end{equation}
Write $\tilde{S}_{\xi,\tau}(x)=b_{\xi+\sigma,\tau}(x)$. The last expression in (\ref{3.53}) can then be 
written
\begin{align}
&\int_{\tilde{\sigma}}^\infty(TS_{\xi_2,\tau_2}(x)-\tilde{S}_{\xi_2,\tau_2}(x))TS_{\xi_1,-\tau_1}(x)\,dx
\notag\\
&+\int_{\tilde{\sigma}}^\infty\int_{\tilde{\sigma}}^\infty
\sum_{r=1}^\infty T^{2r}(x,y)(TS_{\xi_2,\tau_2}(y)-\tilde{S}_{\xi_2,\tau_2}(y))TS_{\xi_1,-\tau_1}(x)\,dxdy
\notag\\
&=\int_{\tilde{\sigma}}^\infty S_{\xi_2,\tau_2}(x)(T^2S_{\xi_1,-\tau_1})(x)\,dx-
\int_{\tilde{\sigma}}^\infty \tilde{S}_{\xi_2,\tau_2}(x)(TS_{\xi_1,-\tau_1})(x)\,dx
\notag\\
&+\sum_{r=1}^\infty\int_{\tilde{\sigma}}^\infty S_{\xi_2,\tau_2}(x)(T^{2r+2}S_{\xi_1,-\tau_1})(x)\,dx
-\sum_{r=1}^\infty\int_{\tilde{\sigma}}^\infty \tilde{S}_{\xi_2,\tau_2}(x)(T^{2r+1}S_{\xi_1,-\tau_1})(x)\,dx
\notag\\
&=\int_{\tilde{\sigma}}^\infty\int_{\tilde{\sigma}}^\infty S_{\xi_2,\tau_2}(x)R(x,y)S_{\xi_1,-\tau_1}(y)\,dxdy
\notag\\
&-\int_{\tilde{\sigma}}^\infty\int_{\tilde{\sigma}}^\infty \tilde{S}_{\xi_2,\tau_2}(x)T(x,y)S_{\xi_1,-\tau_1}(y)\,dxdy
\notag\\
&-\int_{\tilde{\sigma}}^\infty\int_{\tilde{\sigma}}^\infty\int_{\tilde{\sigma}}^\infty 
\tilde{S}_{\xi_2,\tau_2}(x)R(x,y)T(y,z)S_{\xi_1,-\tau_1}(y)\,dxdydz.
\notag
\end{align}
If we add $\tilde{A}(\tau_1,\xi_1+\tau_1^2+\sigma,\tau_2,\xi_2+\tau_2^2+\sigma)$ to this we get exactly the right hand side of (\ref{1.29}).
\end{proof}

\section{Auxiliary results}\label{sect4}

In this section we will prove some results used in the previous sections.

\begin{lemma}\label{lem4.1}

\begin{itemize}
\item[(a)] The kernel $M^{(z,w)}$ given by (\ref{2.40}) and the kernel $\hat{M}^{(z,w)}$ defined by (\ref{2.53})
define finite rank operators on $L^2[1,\infty)$.
\item[(b)] We have that $\det(I-M_0)_{L^2[1,\infty)}>0$ and
$\det(I-\hat{M}_0)_{L^2[1,\infty)}>0$, where $M_0$ is given by (\ref{1.7}) and $\hat{M}_0$ is given
by (\ref{1.12.6}).
\item[(c)] For $z,w$ in a compact subset of the left half plane, there are constants $C,\epsilon>0$ so that
\begin{equation}
|M^{(z,w)}(x,y)|\le Ce^{-\epsilon (x+y)},
\notag
\end{equation}
and the analogous statement holds for $\hat{M}^{(z,w)}$ with $z,w$ in a compact subset
of the right half plane.
\end{itemize}
\end{lemma}

\begin{proof} We will prove the statements for $M^{(z,w)}$ and $M_0$. The proofs for $\hat{M}^{(z,w)}$ and $\hat{M}_0$
are analogous. 

(a) We see from (\ref{2.60}) that it is enough to prove that $M_0$ has finite rank. If we use the residue theorem in (\ref{2.40})
with $(z,w)=(0,0)$ we see that
\begin{align}
M_0(x,y)&=\sum_{r=1}^m\sum_{s=1}^n e^{ax\nu_s-ay\nu_{n+r}}\prod_{j=1}^m(\nu_s-\nu_{n+j})
\prod_{j=1, j\neq s}^n(\nu_s-\nu_j)^{-1}
\notag\\
&\times\prod_{j=1}^n(\nu_{n+r}-\nu_j)\prod_{j=1,j\neq r}^m(\nu_{n+r}-\nu_{n+j})^{-1}\frac 1{\nu_s-\nu_{n+r}}.
\notag
\end{align}
From this formula we see that $M_0$ has finite rank.

(b) It follows from (\ref{2.5}), (\ref{2.12}) and (\ref{2.31}) that
\begin{equation}
s_{<K^m>}(x)=\prod_{j=1}^m e^{a\nu_{n+j}} D_K[g_K(\zeta)]=\prod_{j=1}^m e^{a\nu_{n+j}}
\prod_{i=1}^n\prod_{j=1}^m(1-\gamma_i\delta_j)^{-1}\det(I-\mathcal{K}_0)_{\bar{\ell}^2(K)},
\notag
\end{equation}
where $x_j=\exp(a\nu_j/K)$, $\exp(a\nu_i/K)$, $\gamma_i=\exp(a\nu_i/K)$ and $\delta_j=\exp(-a\nu_{n+j}/K)$.
Note that
\begin{equation}
\prod_{i=1}^n\prod_{j=1}^m(1-\gamma_i\delta_j)^{-1}\sim \frac{K^{mn}}{a^{mn}}\prod_{i=1}^n\prod_{j=1}^m(\nu_{n+j}-\nu_i)^{-1}
\notag
\end{equation}
as $K\to\infty$. We want to show that 
\begin{equation}\label{1}
\lim_{K\to\infty} K^{-mn}s_{<K^m>}(x)>0.
\end{equation}
Let us use the combinatorial formula for the Schur polynomial
\begin{equation}
s_{<K^m>}(x)=\sum_{T\,;\,\text{sh\,}(T)=<K^m>} x^T.
\notag
\end{equation}
Since $\nu_1\le \nu_j$, $1\le j\le n+m$ we see that $x^T\ge e^{am\nu_1}$ if $\text{sh\,}(T)=<K^m>$ and hence
\begin{equation}
s_{<K^m>}(x)\ge e^{am\nu_1}s_{<K^m>}(1^{n+m}).
\notag
\end{equation}
Now,
\begin{equation}
s_\lambda(1^{n+m})=\prod_{1\le i<j\le n+m}\frac{\lambda_i-\lambda_j+j-i}{j-i}
\notag
\end{equation}
and $\ell(\lambda)=m$. Thus,
\begin{equation}
s_{<K^m>}(1^{n+m})=\prod_{i=1}^m\prod_{j=m+1}^{m+n}\frac{K+j-i}{j-i}\sim cK^{mn},
\notag
\end{equation}
as $K\to\infty$, where $c>0$. This proves (\ref{1}).

(c) The inequality follows by a direct estimation of the integral in the right hand side of (\ref{2.40}) using
the fact that $\gamma_{1,w}$ lies strictly in the open left half plane and that $\gamma_2$ lies strictly in the open
right half plane.

\end{proof}

In order to prove the estimates we need for $F(w,k)$ we will use the following lemma.

\begin{lemma}\label{lem4.2}. Let $x_j=\exp(a\nu_j/K)$ and $x_j^{(k)}=\exp(a\nu_j^{(k)}/K)$, $1\le j\le n+m$. Then,
\begin{equation}\label{2.17'}
|s_{<K^m>}(x^{(k)})|\le e^{a(|\re w|+|\nu_k|)}|s_{<K^m>}(x)|
\end{equation}
for $1\le k\le N$.
\end{lemma}

\begin{proof} Let $y_1,\dots,y_N$ be $x_1^{(k)},\dots,x_N^{(k)}$ ordered so that $y_N=x_k^{(k)}=e^{aw/K}$, 
and $y_1,\dots,y_{N-1},z_N$ be $x_1,\dots,x_N$
ordered so that $z_N=x_k=e^{a\nu_k/K}$. By the symmetry of the Schur polynomial and a well known identity we have
\begin{equation}\label{2.17''}
s_{<K^m>}(x^{(k)})=s_{<K^m>}(y_1,\dots,y_N)=\sum_{\mu}s_\mu(y_1,\dots,y_{N-1})s_{<K^m>/\mu}(y_N).
\end{equation}
In order for $s_{<K^m>/\mu}(y_N)$ to be $\neq 0$ we must have $\mu=(K,\dots,K,r,0,\dots)$ with $m-1$ elements equal to $K$ and $0\le r\le K$. Hence
\begin{equation}\label{2.17:3}
s_{<K^m>/\mu}(y_N)=y_N^{K-r}=\left(\frac{y_N}{z_N}\right)^{K-r}s_{<K^m>/\mu}(z_N).
\end{equation}
Now,
\begin{equation}
\left|\left(\frac{y_N}{z_N}\right)^{K-r}\right|\le\left|e^{(K-r)a(w-\nu_k)/K}\right|=e^{a(\re w-\nu_k)(K-r)/K}\le e^{a(|\re w|+|\nu_k|)}.
\end{equation}
Inserting (\ref{2.17:3}) into (\ref{2.17''}) and using this estimate we obtain (\ref{2.17'}).
\end{proof}

We can now establish the estimates that were used in section \ref{sect2.2}. 
From (\ref{2.17'}) and (\ref{2.13}) we obtain
\begin{equation}\label{2.17:4}
|F_K(w,k)|\le  e^{a(|\re w|+|\nu_k|)}
\end{equation}
for $1\le k\le n$. By (\ref{2.13}), (\ref{2.15.2}) and (\ref{2.17'}) we obtain
\begin{equation}\label{2.17:5}
|F_K(w,k)|\le  e^{2a(|\re w|+|\nu_k|)}
\end{equation}
for $n<k\le N$.

Let us also note the following estimate. 
Using the inequality $|e^z-1|\le |z|e^{|\re z|}$ for $z\in\mathbb{C}$ we see that there is a constant $C$ independent of
$K$ such that
\begin{equation}\label{2.17:6}
\left|\prod_{j=1,j\neq k}^{n+m}\frac{e^{aw/K}-e^{a\nu_j/K}}{e^{a\nu_k/K}-e^{a\nu_j/K}}\right|\le C\prod_{j=1,j\neq k}^{n+m}|w-\nu_j|e^{a|\re w|}.
\end{equation}

Finally we give the proof of lemma \ref{lem2.2}.

\begin{proof} (\it of lemma \ref{lem2.2})\rm. We know by lemma \ref{lem4.1} that $M_0$ is finite rank 
operator so we can write
\begin{equation}
M_0=\sum_{j=1}^p \phi_j\otimes \psi_j,
\notag
\end{equation}
for some $p$, where $\phi_j,\psi_j\in L^2[1,\infty)$. Write
\begin{equation}
c_1=\int_{\Gamma_{c_1}} G_{v,t}(w)b_1^w\,dw,\quad c_2=\int_{\gamma_{1}} F_{u,s}(z)b_2^z\,dz 
\notag
\end{equation}
and let $<\,,\,>$ denote the inner product on $L^2[1,\infty)$
We have that
\begin{align}\label{4.10}
&\det(I-M_0+(w-z)b_1^w\otimes b_2^z)=\det
\left(
\begin{matrix}
 \delta_{jk}-<\phi_j,\psi_k> & <\phi_j,b_2^z>\\
 <(w-z)b_1^w,\psi_k>   &1+<(w-z)b_1^w,b_2^z>
\end{matrix}\right)
\notag\\
&=\det\left(
\begin{matrix}
 \delta_{jk}-<\phi_j,\psi_k> & 0\\
 <(w-z)b_1^w,\psi_k>   &1
\end{matrix}
\right)
+\det
\left(
\begin{matrix}
 \delta_{jk}-<\phi_j,\psi_k> & <\phi_j,b_2^z>\\
 <(w-z)b_1^w,\psi_k>   &<(w-z)b_1^w,b_2^z>
\end{matrix}
\right)
\notag\\
&=\det(I-M_0)+(w-z)
\det
\left(
\begin{matrix}
 \delta_{jk}-<\phi_j,\psi_k> & <\phi_j,b_2^z>\\
 <b_1^w,\psi_k>   &<b_1^w,b_2^z>
\end{matrix}
\right).
\end{align}
Here the determinants are of size $(p+1)\times (p+1)$.
It follows that
\begin{align}
&\int_{\gamma_1}\,dz\int_{\Gamma_{c_1}}\, dw F_{u,s}(z)G_{v,t}(w)\det
\left(
\begin{matrix}
 \delta_{jk}-<\phi_j,\psi_k> & <\phi_j,b_2^z>\\
 <b_1^w,\psi_k>   &<b_1^w,b_2^z>
\end{matrix}
\right)
\notag\\
&=\det
\left(
\begin{matrix}
 \delta_{jk}-<\phi_j,\psi_k> & <\phi_j,c_2>\\
 <c_1,\psi_k>   &<c_1,c_2>
\end{matrix}
\right)
\notag\\
&=\det
\left(
\begin{matrix}
 \delta_{jk}-<\phi_j,\psi_k> & <\phi_j,c_2>\\
 <c_1,\psi_k>   &1+<c_1,c_2>
\end{matrix}
\right)-
\det
\left(
\begin{matrix}
 \delta_{jk}-<\phi_j,\psi_k> & 0\\
 <c_1,\psi_k>   &1
\end{matrix}
\right)
\notag\\
&=\det(I-M_0+c_1\otimes c_2)-\det(I-M_0).
\notag
\end{align}

Combining this with (\ref{4.10}) proves the formula (\ref{2.64}).
\end{proof}

%\medskip
%\noindent
%{\bf Acknowledgement}: 


\begin{thebibliography}{99}

\itemsep=\smallskipamount


\bibitem{AvM05}
M.~Adler and P.~van Moerbeke, \emph{PDEs for the joint distributions of the Dyson, Airy and Sine processes},
Annals of Probability, \textbf{33} (2005),1326--1361

\bibitem{AvMD08}
M.~Adler, J.~Del{\'e}pine, and P.~van Moerbeke, \emph{Dyson's nonintersecting
  brownian motions with a few outliers}, Comm. Pure Appl. Math. \textbf{62}
  (2010), 334--395.

\bibitem{AFvM08}
M.~Adler, P.L. Ferrari, and P.~van Moerbeke, \emph{Airy processes with
  wanderers and new universality classes}, Ann. Probab. \textbf{38} (2008),
  714--769.
\bibitem{AFvM10}
M.~Adler, P.L. Ferrari, and P.~van Moerbeke, \emph{Non-intersecting random walks in the neighborhood of
a symmetric tacnode}, arXiv:1007.1163 (2010).

\bibitem{AJvM11}
M.~Adler, K.~Johansson, and P.~van Moerbeke, \emph{Double Aztec diamonds}, in preparation.

\bibitem{AOvM10}
M.~Adler, N.~Orantin, and P.~van Moerbeke, \emph{Universality for the Pearcey process}, Physica D \textbf {239} (2010), 924--941.

\bibitem{AvMV10}
M.~Adler, P.~van Moerbeke and D. Vanderstichelen, \emph{Non-interecting Brownian Motions leaving from and going to
several points}, arXiv:1005.1303 (2010).

\bibitem{BW00}
E. L. Basor and H. Widom, \emph{On a Toeplitz determinant identity of Borodin and Okounkov},  
Integral Equations Operator Theory,  \textbf{37} (2000), 397--401.

\bibitem{BD10}
A.~Borodin and M.~Duits, \emph{Limits of determinantal processes near a
  tacnode}, arXiv:0911.1980. To appear in Ann. Inst. H. Poincar\'e.

\bibitem{BO99}
A.~Borodin and A.~Okounkov, \emph{{A Fredholm determinant formula for Toeplitz
  determinants}}, Integr. Equ. Oper. Theory \textbf{37} (2000), 386--396.

\bibitem{BH97}
E. Br\'ezin and S. Hikami, \emph{Extension of level-spacing universality}, Phys. Rev. E \textbf{56} (1997), 264--269.

\bibitem{BH98}
E. Br\'ezin and S. Hikami, \emph{Universal singularity at the closure of a gap}, Phys. Rev. E \textbf{57} (1998)., 
4140--4149

\bibitem{De99}
P.~Deift, T. Kriecherbauer, K. T-R. McLaughlin, S. Venakides, X. Zhou, \emph{Strong asymptotics for orthogonal 
polynomials with respect to exponential weights}, Commun. Pure Appl. math., \textbf{52} (1999), 1491--1552

\bibitem{DK08}
S.~Delvaux and A.~Kuijlaars, \emph{A phase transition for non-intersecting
  Brownian motions, and the Painlev\'e equation}, Int. Math. Res. Not. IMRN2009, 3639--3725.

\bibitem{DK09}
S.~Delvaux and A.~Kuijlaars, \emph{A graph based equilibrium problem for the 
limiting distribution of non-interesecting Brownian motions at low temperature},
Constr. Approx., \textbf{32} (2010), 467--512.

\bibitem{DKZ10}
S.~Delvaux, A.~Kuijlaars, and L.~Zhang, \emph{Critical behavior of
  non-intersecting Brownian motions at a tacnode}, Commu. Pure Appl. Math., \textbf{64} (2011), 1305--1383

\bibitem{Er10}
L.~Erd\"os, S. P\'ech\'e, J. A. Ram\'irez, B. Schlein, H.-T. Yau, \emph{Bulk universality for Wigner matrices},
Comm. Pure Appl. Math.  \textbf{63} (2010), 895--925.

\bibitem{EM97}
B.~Eynard and M.L. Mehta, \emph{Matrices coupled in a chain. {I}. {E}igenvalue
  correlations}, J. Phys. A \textbf{31} (1998), 4449--4456.

\bibitem{FS03}
P.L. Ferrari and H.~Spohn, \emph{Step fluctations for a faceted crystal}, J.
  Stat. Phys. \textbf{113} (2003), 1--46.

\bibitem{Fo93} 
P.~Forrester, \emph{The Spectrum edge of random matrix ensembles}, Nucl. Phys. B
\textbf{402} (1993), 709--728.

\bibitem{GC79}
J. S. Geronimo and  K. M. Case, \emph{Scattering theory and polynomials orthogonal on the unit circle},
J. Math. Phys.  \textbf{20}  (1979), 299--310.

\bibitem{Jo01}
K.~Johansson, \emph{Universality of the local spacing distribution in
certain Hermitian Wigner matrices}, Commun. Math. Phys. \textbf{215} (2001), 683--705.

\bibitem{Jo03}
K.~Johansson, \emph{Discrete polynuclear growth and determinantal processes},
  Comm. Math. Phys. \textbf{242} (2003), 277--329.

\bibitem{Jo05}
K.~Johansson, \emph{{Non-intersecting, simple, symmetric random walks and the
  extended Hahn kernel}}, Ann. Inst. Fourier \textbf{55} (2005), 2129--2145.

\bibitem{KM59}
S.~Karlin and L.~McGregor, \emph{Coincidence probabilities}, Pacific J.
  \textbf{9} (1959), 1141--1164.

\bibitem{KT07}
M.~Katori and H.~Tanemura, \emph{{Noncolliding Brownian Motion and
  Determinantal Processes}}, J. Stat. Phys. \textbf{129} (2007), 1233--1277.

\bibitem{FN98}
T.~Nagao and P.J. Forrester, \emph{{Multilevel dynamical correlation functions
  for Dysons Brownian motion model of random matrices}}, Phys. Lett. A
  \textbf{247} (1998), 42--46.

\bibitem{OR03}
A.~Okounkov and N.~Reshetikhin, \emph{Correlation function of Schur process
  with application to local geometry of a random 3-dimensional {Y}oung
  diagram}, J. Amer. Math. Soc. \textbf{16} (2003), 581--603.

\bibitem{OR07}
A.~Okounkov and N.~Reshetikhin, \emph{{Random skew plane partitions and the
  Pearcey process}}, Comm. Math. Phys. \textbf{269} (2007), 571--609.

\bibitem{PS02}
M. Pr\"ahofer and H.~Spohn, \emph{Scale invariance of the PNG droplet and
the Airy Process}, J. Stat. Phys. \textbf{108} (2002), 1076--1106.

\bibitem{Sa} B.E.~Sagan, \emph{The symmetric group. Representations, combinatorial algorithms, and symmetric functions}, 
Second edition. Graduate Texts in Mathematics, 203. Springer-Verlag, New York, 2001. 

\bibitem{So} A. Soshnikov, {\em Determinantal random point fields,}
 Russian Math. Surveys, {\bf 55} (2000), 923--975.

\bibitem{TW06}
C.~Tracy and H.~Widom, \emph{{The Pearcey Process}}, Comm. Math. Phys.
  \textbf{263} (2006), 381--400.

\bibitem{Va05}
M. Vanlessen, \emph{Strong asymptotics of Laguerre-type orthogonal polynomials
and applications in random matrix theory}, Constr. Approx., \textbf{25} (2007), 125-175.


\end{thebibliography}
\end{document}